\documentclass[10pt]{amsart}

\usepackage{amssymb,amsmath,amsthm,amscd}

\advance\oddsidemargin by -1cm
\advance\evensidemargin by -1cm 
\newtheorem{theorem}{Theorem}

\newtheorem{proposition}{Proposition}
\newtheorem{corollary}{Corollary}
\newtheorem{lemma}{Lemma}
\newtheorem{definition}{Definition}

\theoremstyle{definition}
\newtheorem{remark}{Remark}

\newcommand{\bdm}{\begin{displaymath}}
\newcommand{\edm}{\end{displaymath}}
\newcommand{\bq}{\begin{equation}}
\newcommand{\eq}{\end{equation}}
\newcommand{\bqn}{\begin{equation*}}
\newcommand{\eqn}{\end{equation*}}

\newcommand{\rn}{\mathbb{R}^n}

\newcommand{\norm}[1]{\left\| #1 \right\|}
\newcommand{\mklm}[1]{\left\{ #1 \right\}}
\newcommand{\eklm}[1]{\left\langle  #1 \right\rangle}

\renewcommand{\d}{\,d}
\newcommand{\N}{{\mathbb N}}

\newcommand{\C}{{\mathbb C}}
\newcommand{\R}{{\mathbb R}}

\newcommand{\B}{{\mathcal B}}
\newcommand{\D}{{\mathcal D}}
\newcommand{\E}{{\mathcal E}}

\newcommand{\X}{{\mathbb X}}
\newcommand{\0}{{\rm 0}}
\renewcommand{\epsilon}{\varepsilon}
\renewcommand{\phi}{\varphi}
\renewcommand{\rho}{\varrho}
\newcommand{\1}{{ \bf  1}}

\newcommand{\Cinft}{{\rm C^{\infty}}}
\newcommand{\CT}{{\rm C^{\infty}_c}}
\newcommand{\Cinftv}{{\rm \dot C^{\infty}}}
\newcommand{\CTv}{{\rm \dot C^{\infty}_c}}

\renewcommand{\L}{{\rm L}}

\renewcommand{\S}{{\mathcal S}}
\newcommand{\G}{{\mathcal G}}
\newcommand{\Sym}{{\rm S}}
\newcommand{\Syms}{{\rm S^{-\infty}}}
\newcommand{\Symsl}{{\rm S^{-\infty}_{la}}}

\newcommand{\GL}{\mathrm{GL}}

\newcommand{\g}{{\bf \mathfrak g}}
\renewcommand{\k}{{\bf \mathfrak k}}
\renewcommand{\a}{{\bf\mathfrak a}}
\newcommand{\m}{{\bf\mathfrak m}}
\newcommand{\n}{{\bf\mathfrak n}}
\newcommand{\p}{{\bf \mathfrak p}}

\newcommand{\U}{{\mathfrak U}}

\newcommand{\Ad}{\mathrm{Ad}\,}
\newcommand{\ad}{\mathrm{ad}\,}
\newcommand{\sgn }{\mathrm{sgn }\,}

\renewcommand{\det}{\mathrm{det}\,}

\renewcommand{\Re}{\mathrm{Re}\,}

\DeclareMathOperator{\supp}{supp}

\DeclareMathOperator{\tr}{tr}
\DeclareMathOperator{\gd}{\partial}

\newcommand{\e}[1]{\,{\mathrm e}^{#1}\,}
\newcommand{\dbar}{{\,\raisebox{-.1ex}{\={}}\!\!\!\!d}}

\begin{document}

\author{Aprameyan Parthasarathy and Pablo Ramacher}
\title[Integral operators on Oshima compactifications of Riemannian symmetric spaces]{Integral operators on the
 Oshima compactification of a  Riemannian symmetric space of non-compact type. Microlocal analysis and  kernel asymptotics }
\address{Aprameyan Parthasarathy and Pablo Ramacher, Fachbereich Mathematik und Informatik, Philipps-Universit\"at Marburg,  
Hans-Meerwein-Str., 35032 Marburg, Germany}
\email{apra@mathematik.uni-marburg.de, ramacher@mathematik.uni-marburg.de}
\thanks{The authors wish to thank Toshio Oshima for valuable conversations on the subject. We are also grateful to Jean-Philippe Anker for 
helpful remarks on the classical heat kernel on Riemannian symmetric spaces. This work was financed by the DFG-grant RA 1370/2-1.}

\begin{abstract}
Let $\X\simeq G/K$ be a Riemannian symmetric space of non-compact type, $\widetilde \X$ its Oshima compactification, and
 $(\pi,\mathrm{C}(\widetilde \X))$ the regular representation of $G$ on  $\widetilde \X$. We study integral operators on $\widetilde \X$ 
of the form $\pi(f)$, where $f$ is a rapidly falling function on $G$, and characterize them within the framework of pseudodifferential 
operators, describing the singular nature of their kernels. In particular, we consider the holomorphic semigroup generated
 by a strongly  elliptic operator associated to the representation $\pi$, as well as  its resolvent, and describe the 
asymptotic behavior of the corresponding semigroup and resolvent kernels.  
\end{abstract}

\maketitle

\tableofcontents

\section{Introduction}

Let $\X$ be a Riemannian symmetric space of non-compact type. Then $\X$ is isomorphic to $G/K$, where $G$ is a connected real 
semisimple Lie group, and $K$ a maximal compact subgroup. Consider further the Oshima compactification \cite{oshima78} $\widetilde \X$ of $\X$,
 a simply connected closed real-analytic manifold on which $G$ acts analytically. The orbital decomposition 
of $\widetilde \X$ is of normal crossing type, and the open orbits are isomorphic to $G/K$, the number of them being equal 
to $2^l$, where $l$ denotes the rank of $G/K$. In this paper, we shall study the invariant integral operators
\bq
\label{eq:1}
 \pi(f)= \int _G f(g) \pi(g) d_G (g),
\eq
where $\pi$ is the regular representation of $G$ on the Banach space $\mathrm{C}(\widetilde \X)$ of continuous 
functions on $\widetilde \X$,  $f$  a smooth, rapidly decreasing function on $G$, and $d_G$ a Haar measure on $G$. 
These operators  play an important role in representation theory, and our interest will be directed towards the 
elucidation of  their microlocal structure within the theory of pseudodifferential operators. Since  the underlying
 group action on $\widetilde \X$ is not transitive, the operators $\pi(f)$ are not smooth,  and the orbit structure
 of $\widetilde \X$ is  reflected in the singular behavior of their Schwartz kernels.
As it turns out, the operators in question can be characterized as pseudodifferential operators belonging to a particular class 
which was first introduced in \cite{melrose} in connection with  boundary problems. In fact, if $\widetilde \X_\Delta$ denotes a component 
in $\widetilde \X$ isomorphic to $G/K$, we prove  that  the  restrictions 
 \begin{equation*}
\pi(f)_{|\overline{\widetilde \X_{\Delta}}}:\CT(\overline{\widetilde \X_{\Delta}}) \longrightarrow
 \Cinft(\overline{\widetilde \X_{\Delta}})
\end{equation*}
of the operators $\pi(f)$ to the manifold
with corners $\overline{\widetilde \X_{\Delta}}$
 are  totally characteristic pseudodifferential operators of class $\L^{-\infty}_b$. A similar 
description of invariant integral operators on prehomogeneous vector spaces was obtained by the second author in \cite{ramacher06}. 
 We then consider the holomorphic semigroup generated by  a strongly elliptic operator $\Omega$ associated to the  regular
 representation $(\pi, \mathrm{C}(\widetilde \X))$ of $G$,   as well as  its resolvent. Since both the holomorphic semigroup and the 
resolvent can be characterized as  operators of the form \eqref{eq:1}, they can be studied  with the previous methods, and relying on the 
theory of elliptic operators on Lie groups \cite{robinson} we  obtain a description of the asymptotic behavior of the semigroup and resolvent kernels
 on $\widetilde \X_\Delta\simeq \X$ at infinity. In the particular case of the Laplace-Beltrami operator on $\X$, 
these questions have been intensively studied before. 
While for the classical heat kernel on $\X$ precise upper and lower bounds were  previously obtained in \cite{anker-ji99} 
using spherical analysis,  a detailed description of the analytic properties of the resolvent of the Laplace-Beltrami 
operator on $\X$ was given in \cite{mazzeo-melrose87}, \cite{mazzeo-vasy05}.

The paper is organized as follows. In Section \ref{Sec:2} we briefly recall those parts of the   structure theory of real semisimple 
Lie groups that are relevant to our purposes. We then describe the $G$-action on the homogeneous spaces $G/P_{\Theta}(K)$, 
where $P_{\Theta}(K)$ is a closed subgroup of $G$ associated naturally to a subset $\Theta$ of the set of simple roots, and 
the corresponding fundamental vector fields. This leads to the definition of the Oshima compactification $\widetilde{\X}$ of the
 symmetric space $\X\simeq G/K$, together with a description of the orbital decomposition of  $\widetilde{\X}$. Since this  decomposition 
is of normal crossing type, it is well-suited for our analytic purposes.  A thorough and unified description of the various 
compactifications of a symmetric space is given in \cite{borel-ji}.  Section \ref{sec:PDO} contains a summary with  some of the basic
facts in the theory pseudodifferential operators needed in the sequel. In particular, the class of totally characteristic 
pseudodifferential operators on a manifold with corners is introduced. Section \ref{Sec:4} is the central part 
of this paper.  By analyzing the orbit structure of the $G$-action on $\widetilde{\X}$, we
are able to elucidate the microlocal structure of the convolution operators $\pi(f)$, and characterize them as 
totally characteristic pseudodifferential operators on the manifold with corners  $\overline{\widetilde{\X}_\Delta}$. This
 leads to a description of the asymptotic behavior of their  Schwartz kernels on $\widetilde \X_\Delta\simeq \X$ at infinity.
In Section \ref{Sec:5}, we consider the holomorphic semigroup $S_\tau$ generated  by the closure $\overline \Omega$ of a strongly
 elliptic differential operator $\Omega$ associated to the representation $\pi$. Since $S_\tau=\pi(K_\tau)$, where $K_\tau(g)$ 
is a smooth and rapidly decreasing function on $G$,  we can apply our previous results to describe 
the Schwartz kernel of $S_\tau$.  The Schwartz kernel of the resolvent $(\lambda \1+ \overline{\Omega})^{-\alpha}$,
 where $\alpha > 0$, and $\Re \lambda$ is sufficiently large, can be treated similarly, but is more subtle 
due to the singularity of the corresponding group kernel $R_{\alpha,\lambda}(g)$ at the identity.
 
\newpage

\section{The Oshima compactification of a Riemannian symmetric space}
\label{Sec:2}

Let $G$ be a connected real semisimple Lie group with finite  centre and Lie algebra $\g$, and denote by $\langle X,Y\rangle = \tr \, (\ad X\circ \ad Y)$ the \emph{Cartan-Killing form} on $\g$. Let $\theta$ be the Cartan involution of $\g$, and 
$$\g = \k\oplus\p$$
 the Cartan decomposition of $\g$ into the eigenspaces of  $\theta$, corresponding to the eigenvalues  $+1$ and $-1$ , respectively, and put $\langle X,Y\rangle _\theta:=-\langle X,\theta Y\rangle $. Note that the Cartan decomposition is orthogonal with respect to $\langle,\rangle_{\theta}$. Consider further a maximal Abelian subspace $\a$ of $\p$. Then $\ad(\a)$ is a commuting family of self-adjoint operators on $\g$. Indeed, for  $X,Y,Z\in\g$ one computes
   \begin{align*}
   \langle \ad X(Z),Y\rangle _\theta&=-\langle [X,Z],\theta Y\rangle =-\langle Z,[\theta Y,X]\rangle =-\langle Z,\theta [Y,\theta X]\rangle =\langle Z,[Y,\theta X]\rangle _\theta\\&=\langle Z,-[\theta X,Y]\rangle _\theta=\langle Z,-\ad\theta X(Y)\rangle _\theta.
   \end{align*}
  Therefore  $-\ad\theta X$ is the adjoint of $\ad X$ with respect to $\langle ,\rangle _\theta$. So, if we take $X\in\p$, the -1 eigenspace of $\theta$,  $\ad X$ is self-adjoint with respect to  $\langle ,\rangle _\theta$.  The dimension $l$ of $\a$ is called the \emph{real rank of $G$} and \emph{the rank of the 
symmetric space $G/K$}. Next, one defines for each $\alpha\in\a^*$, the dual of $\a$, the simultaneous eigenspaces $\g^\alpha=\{X\in\g:[H,X]=\alpha(H)X \, \text{for all } H\in\a\}$ of $\ad(\a)$. A functional $0 \not =\alpha\in\a^\ast $ is called a  \emph{(restricted) root} of $(\g,\a)$ if $\g^\alpha\neq\{0\}$, and setting  $\Sigma=\{\alpha\in\a^*:\alpha\neq0,\g^\alpha\neq\{0\}\}$, we obtain the decomposition 
\bqn
\g=\m\oplus \a\oplus \bigoplus_{\alpha\in \Sigma}\g^\alpha,
\eqn
where $\m$ is the centralizer of $\a$ in $\k$. Note that this decomposition is  orthogonal with respect to $ \langle \cdot ,\cdot\rangle _\theta$.  With respect to an ordering of $\a^\ast$, let  $\Sigma^{+}=\{\alpha\in\Sigma:\alpha> 0\}$ denote the  \emph{set of positive roots}, and $\Delta=\{\alpha_1,\dots ,\alpha_l\}$  the \emph{set of simple roots}.  Let $\rho=\frac{1}{2}\Sigma_{\alpha\in\Sigma^{+}}\alpha$, and put $m(\alpha)=$ dim $\g^\alpha$ which is, in general,  greater than $1$. 
Define $\n^+=\bigoplus_{\alpha\in\Sigma^{+}}\g^\alpha, \, \n^-=\theta(\n^+)$, and write  $K, A, N^+$ and $N^- $ for  the analytic subgroups of $G$ corresponding to $\k,\, \a,\,  \n^+$, and $ \n^-$, respectively. The \emph{Iwasawa decomposition} of $G$ is then given by 
\bqn
G=KAN^{\pm}.
\eqn
Next, let  $M=\{k\in K:\Ad(k)H=H \text{ for all } H\in\a\}$ be the centralizer of $\a$ in $K$ and $M^*=\{k\in K:\Ad(k)\a\subset\a\}$ the normalizer of $\a$ in $K$. The quotient $W={M^*}/{M}$ is the \emph{Weyl group} corresponding to $(\g,\a)$, and acts on $\a$ as a group of linear transformations via the adjoint action.  Alternatively, $W$ can be characterized as follows.  For each $\alpha_i \in \Delta $,  define a reflection in $\a^*$ with respect to the Cartan-Killing form $\langle \cdot  , \cdot \rangle $ by 
 \bqn
 w_{\alpha_i}:\lambda\mapsto \lambda -2\alpha_{i}\langle \lambda,\alpha_{i}\rangle /\langle \alpha_i,\alpha_i\rangle, 
 \eqn
 where $\langle \lambda,\alpha_{}\rangle=\langle H_\lambda, H_\alpha \rangle.$ Here $H_\lambda$ is the unique element in $\a$ corresponding to a given $\lambda\in\a^\ast$, and determined by the non-degeneracy of the Cartan-Killing form. One can then identify the Weyl group $W$ with the group generated by the reflections $\{w_{\alpha_i}:\alpha_i\in\Delta\}$. For a subset $\Theta$ of $\Delta$, let now $W_{\Theta}$ denote the subgroup of W generated by  reflections corresponding to elements in $\Theta$, and define 
 \bqn
 P_{\Theta}=\bigcup_{w\in W_{\Theta}}Pm_{w}P,
 \eqn
 where $m_w$ denotes a representative of $w$ in $M^\ast$, and $P=MAN^+$ is  a minimal parabolic subgroup.  It is then a classical result in   the theory of parabolic subgroups  \cite{warner72} that, as $\Theta$ ranges over the subsets of $\Delta$, one obtains all the parabolic subgroups of $G$ containing $P$. In particular, if $\Theta = \emptyset$,  $P_\Theta=P$. Let us now introduce for  $\Theta\subset\Delta$ the subalgebras
 \bqn
 \a_{\Theta}=\{H\in\a:\alpha(H)=0 \, \text{ for all }   \alpha\in\Theta\}, \qquad \a(\Theta)=\{H\in\a:\langle H,X\rangle _{\theta}=0 \text{ for all } X\in\a_{\Theta}\}.
 \eqn
Note that, when restricted to the $+1$ or the $-1$ eigenspace of $\theta$, the orthogonal complement of a subspace with respect to  $\langle \cdot ,\cdot \rangle $ is the same as its orthogonal complement with respect to  $\langle \cdot ,\cdot \rangle _{\theta}$. We further define 
\begin{align*}
\n_{\Theta}^{+}&=\sum_{\alpha\in\Sigma^{+}\setminus\langle \Theta\rangle ^{+}}\g^{\alpha}, 
\hspace{4cm} \n_{\Theta}^{-}=\theta(\n_{\Theta}^{+}),\\
 \n^+(\Theta)&=\sum_{\alpha\in\langle \Theta\rangle ^{+}}\g^{\alpha}, \hspace{4cm}  \n^-(\Theta)=\theta(\n^+(\Theta)), \\
 \m_{\Theta}&=\m+\n^+(\Theta)+\n^-(\Theta)+\a(\Theta),\hspace{1.1cm} \m_{\Theta}(K)=\m_{\Theta}\cap\k,
\end{align*} 
 where $\left\langle \Theta\right\rangle ^{+}=\Sigma^{+}\cap\sum_{\alpha_{i}\in\Theta}\R\alpha_{i}$, and denote by $A_{\Theta},A(\Theta), \,N_{\Theta}^{\pm},N^{\pm}(\Theta),M_{\Theta,0}$, and $M_{\Theta}(K)_{0}$ the corresponding connected analytic subgroups of $G$, obtaining the decompositions $A=A_\Theta A(\Theta)$ and $N^\pm =N_\Theta^\pm N(\Theta)^\pm$, the second being a semi-direct product. Let next $M_{\Theta}=MM_{\Theta,0}, \, M_{\Theta}(K)=MM_{\Theta}(K)_{0}$. One has  the \emph{Iwasawa decompositions}
 \bqn 
 M_{\Theta}=M_{\Theta}(K)A(\Theta)N^{\pm}(\Theta),
 \eqn
 and the \emph{Langlands decompositions}
  \begin{align*}
P_{\Theta}&=M_{\Theta}A_{\Theta}N_{\Theta}^{+} =M_{\Theta}(K)AN^{+}.
\end{align*}
In particular, $P_\Delta= M_\Delta =G$, since $m_\Delta = \m \oplus\a \oplus \bigoplus_{\alpha \in \Sigma} \g^\alpha$, and $\a_\Delta, \n_\Delta ^+$ are trivial. 
One then defines  
$$P_{\Theta}(K)=M_{\Theta}(K)A_{\Theta}N_{\Theta}^{+}.$$ 
$P_\Theta(K)$ is a closed subgroup, and $G$ is a union of the open and dense submanifold $N^- A(\Theta) P_\Theta(K)=N^-_\Theta P_\Theta$, and submanifolds of lower dimension, see \cite{oshima78}, Lemma 1.
For $\Delta=\{ \alpha_{1},\dots ,\alpha_{l}\}$, let next $\{H_{1},\dots ,H_{l}\}$ be the basis of $\a$, dual to $\Delta$,  i.e. $\alpha_{i}(H_{j})=\delta_{ij}$. Fix a basis $\{X_{\lambda,{i}}:1\leq i \leq m(\lambda)\}$ of $\g^{\lambda}$ for each $\lambda\in\Sigma^{+}$. Clearly,
\bqn 
[H, - \theta X_{\lambda,i}]=-\theta [\theta H, X_{\lambda,i}]=-\lambda(H) (-\theta X_{\lambda,i}), \qquad H \in \a,
\eqn
so that setting $X_{-\lambda,i}=-\theta(X_{\lambda,{i}})$ one obtains a basis $\{X_{-\lambda,{i}}:1\leq i \leq m(\lambda)\}$ of  $\g^{-\lambda} \subset \n^-$. One now has the following lemma due to Oshima.
\begin{lemma} 
\label{lemma:fundvec}
  Fix an element $g\in G$, and identify $N^{-}\times A(\Theta)$ with an open dense submanifold of the homogeneous space $G/P_{\Theta}(K)$  by the map $(n,a)\mapsto gna P_{\Theta}(K)$. For $Y\in\g$, let $Y_{|G/P_{\Theta}(K)}$ be the fundamental vector field corresponding to the action of the one-parameter group $\exp (sY),s\in\R,$ on $G/P_{\Theta}(K)$. Then, at any point $p=(n,a)\in N^{-}\times A(\Theta)$, we have 
\begin{align*}
(Y_{|G/P_{\Theta}(K)})_{p}&=\sum_{\lambda\in \Sigma^+}\sum_{i=1}^{m(\lambda)} c_{-\lambda,i}(g,n)(X_{-\lambda,i})_p+ \sum_{\lambda\in \langle \Theta\rangle ^+} \sum_{i=1}^{m(\lambda)}c_{\lambda,i}(g,n)e^{-2\lambda( \log a)}(X_{-\lambda,i})_p \\ &+ \sum_{\alpha_i\in\Theta} c_{i}(g,n)(H_i)_p
\end{align*}
with the identification $T_{n}N^{-}\bigoplus T_{a}(A(\Theta))\simeq T_{p}(N^{-}\times A(\Theta))\simeq T_{gnaP_{\Theta}(K)}G/P_{\Theta}(K)$. The coefficient functions $c_{\lambda,{i}}(g,n),c_{-\lambda,i}(g,n),c_{i}(g,n)$ are real-analytic, and are determined by the equation
\bq
\label{eq:2211}
\Ad^{-1}(gn)Y=\sum_{\lambda\in \Sigma^+}\sum_{i=1}^{m(\lambda)}(c_{\lambda,i}(g,n)X_{\lambda,i}+c_{-\lambda,i}(g,n)X_{-\lambda,i})+\sum_{i=1}^{l}c_{i}(g,n)H_{i} \mod \m.
\eq
\end{lemma}
\begin{proof} 
Due to its importance, and for the convenience of the reader, we shall  give a detailed proof of the lemma, following the original proof given in   \cite{oshima78}, Lemma 3. Let $s\in \R$, and assume that $|s|$ is small. According to  the direct sum decomposition $\g=\n^-\oplus \a\oplus\n^+\oplus\m$ one has for an arbitrary $Y\in \g$ 
\begin{equation}
\label{eq:2}
(gn)^{-1}\exp(sY)gn=\exp N_1^-(s)\exp A_1(s)\exp N_1^+(s)\exp M_1(s) ,
\end{equation}
where $N_1^-(s)\in \n^-$, $A_1(s)\in \a$, $N_1^+(s)\in \n^+$, and $ M_1(s)\in \m$. The action of $\exp(sY)$ on the homogeneous space $G/P_{\Theta}(K)$ is therefore given by 
\begin{align*}
\exp(sY)gnaP_{\Theta}(K)&=gn\exp N_1^-(s)\exp A_1(s)\exp N_1^+(s)\exp M_1(s)aP_{\Theta}(K)\\ & = gn\exp N_1^-(s)\exp A_1(s)\exp N_1^+(s)a\exp M_1(s)P_{\Theta}(K) \\ &=gn\exp N_1^-(s)\exp A_1(s)\exp N_1^+(s)aP_{\Theta}(K),
\end{align*} 
since $M$ is the centralizer of $A$ in $K$, and $\exp M_1(s) \in M M_\Theta(K)_0 \subset P_\Theta(K)$. The Lie algebra of $P_{\Theta}(K)$ is $\m_{\Theta}(K)\oplus \a_{\Theta}\oplus \n^+_{\Theta}$, which we shall henceforth  denote  by $\p_\Theta(K)$. Using the decomposition $\g=\n^-\oplus\a(\Theta)\oplus \p_\Theta(K)$ we see that 
\begin{equation}
\label{eq:3}
a^{-1}\exp N_1^+(s)a=\exp N_2^-(s)\exp A_2(s)\exp P_2(s),
\end{equation} 
where $N_2^-(s)\in \n^-$, $A_2(s)\in\a(\Theta)$, and $P_2(s)\in\p_\Theta(K)$. From this we obtain that 
\begin{align*}
gn\exp N_1^-(s)&\exp A_1(s)\exp N_1^+(s)aP_{\Theta}(K)\\& =gn\left(\exp N_1^-(s)\exp A_1(s)a\exp N_2^-(s)\right)\exp A_2(s)\exp P_2(s)P_{\Theta}(K)\\&= gn\left(\exp N_1^-(s)\exp A_1(s)a\exp N_2^-(s)a^{-1}\right)a\exp A_2(s)P_{\Theta}(K).
\end{align*}
Noting that $[\a,\n^-]\subset\n^-$ one deduces  the equality $\exp N_1^-(s)\exp A_1(s)a\exp N_2^-(s)a^{-1}\exp A_1(s)^{-1}=\exp N_3^-(s)\in N^-$, and consequently
\begin{equation}
\label{eq:4}
\exp N_1^-(s)\exp A_1(s)a\exp N_2^-(s)a^{-1}=\exp N_3^-(s)\exp A_1(s),
\end{equation}
which in turn yields 
\begin{align*}
gn\exp N_1^-(s)\exp A_1(s)\exp N_1^+(s)aP_{\Theta}(K)&=gn\exp N_3^-(s)\exp A_1(s)a\exp A_2(s)P_{\Theta}(K)\\&=gn\exp N_3^-(s)a\exp (A_1(s)+A_2(s))P_{\Theta}(K).
\end{align*}
The action of $\g$ on  $G/ P_{\Theta}(K)$ can therefore be characterized as 
\begin{equation}
\label{eq:action}
\exp(sY)gnaP_{\Theta}(K)=gn\exp N_3^-(s)a\exp (A_1(s)+A_2(s))P_{\Theta}(K).
\end{equation}
Set ${dN_i^-(s)}/{ds}|_{s=0}= N_i^-$, ${dN_1^+(s)}/{ds}|_{s=0}= N_1^+$, ${dA_i(s)}/{ds}|_{s=0}= A_i$, and ${dP_2(s)}/{ds}|_{s=0}= P_2$, where  $i=1,2$, or $3$. By differentiating equations \eqref{eq:2}-\eqref{eq:4} at $s=0$ one computes
\begin{align}
\label{eq:Ad}
\Ad^{-1}(gn)Y&=N_1^-+A_1+N_1^+ \qquad \text{mod} \quad \m, \\
 \label{eq:Ad2} 
 \Ad^{-1}(a)N_1^+&=N_2^-+A_2 +P_2,\\
 N_1^-+\Ad(a)N_2^-&=N_3^-.
 \end{align}
 In what follows, we express  $N_1^\pm \in \n^\pm$ in terms of the basis  of $\n^\pm$, and $A_1$ in terms of the one of $\a$, as
 \begin{align*}
  N_1^\pm&=\sum_{\lambda \in \Sigma^+}\sum_{i=1}^{m(\lambda)}c_{\pm\lambda,i}(g,n)X_{\pm\lambda,i},\\
 A_1&=\sum_{i=1}^lc_i(g,n)H_i=\sum_{\alpha_i \in \Theta}c_i(g,n)H_i\quad \text{ mod} \, \a_\Theta.
 \end{align*} 
 For a fixed  $X_{\lambda,i}$ one has   $[H,X_{\lambda,i}]=\lambda(H)X_{\lambda,i}$ for all $H\in \a$. Setting $H=-\log a$, $a \in A$,  we get  $\ad(-\log a)X_{\lambda,i} =-\lambda(\log a)X_{\lambda,i}$. By exponentiating we obtain $e^{\ad(-\log a)}X_{\lambda,i}=e^{-\lambda(\log a)}X_{\lambda,i}$, which together with the relation $e^{\ad(-\log a)}=\Ad(\exp(-\log a))$ yields
  $$\Ad^{-1}(a)X_{\lambda,i}=e^{-\lambda(\log a)}X_{\lambda,i}.$$
Analogously, one has $[H,X_{-\lambda,i}]=\theta[ \theta H, -X_{\lambda,i}]= - \lambda(H)X_{-\lambda,i}$ for all $H\in \a$, so that 
\bq
\label{eq:7}
 \Ad^{-1}(a)X_{-\lambda,i}=e^{\lambda(\log a)}X_{-\lambda,i}.
 \eq
  We therefore arrive at 
   \begin{align*}\Ad^{-1}(a)X_{\lambda,i}&=e^{-\lambda(\log a)}(X_{\lambda,i}-X_{-\lambda,i})+e^{-\lambda(\log a)}X_{-\lambda,i}\\ &=e^{-\lambda(\log a)}(X_{\lambda,i}-X_{-\lambda,i})+e^{-2\lambda(\log a)}\Ad^{-1}(a)X_{-\lambda,i}.
  \end{align*}
Now, since $\theta(X_{\lambda,i}-X_{-\lambda,i})=\theta(X_{\lambda,i})-\theta(X_{-\lambda,i})=-X_{-\lambda,i}-(-X_{\lambda,i})=X_{\lambda,i}-X_{-\lambda,i},$ we see that $X_{\lambda,i}-X_{-\lambda,i}\in \k$. Consequently, if $\lambda$ is in $\left\langle \Theta\right\rangle ^{+}$, one deduces that 
 $X_{\lambda,i}-X_{-\lambda,i} \in \left(\m+\n^+(\Theta)+\n^-(\Theta)+\a(\Theta)\right)\cap\k=\m_{\Theta}(K).$ 
 On the other hand, if $\lambda$ is in $\Sigma^+-\left\langle \Theta\right\rangle ^{+}$, then  $\Ad^{-1}(a)X_{\lambda,i}=e^{-\lambda(\log a)}X_{\lambda,i}$ belongs to $\n_\Theta^+$.
 Collecting everything we obtain
\begin{align*}
\Ad^{-1}(a)N_1^+&=\sum_{\lambda \in \Sigma^+}\sum_{i=1}^{m(\lambda)}c_{\lambda,i}(g,n)\Ad^{-1}(a)X_{\lambda,i} 
\\&=\sum_{\lambda \in \left\langle \Theta\right\rangle ^{+}}\sum_{i=1}^{m(\lambda)}c_{\lambda,i}(g,n)\Ad^{-1}(a)X_{\lambda,i}+\sum_{\lambda \in \Sigma^+-\left\langle \Theta\right\rangle ^{+}}\sum_{i=1}^{m(\lambda)}c_{\lambda,i}(g,n)\Ad^{-1}(a)X_{\lambda,i}
 \\&=\sum_{\lambda \in \left\langle \Theta\right\rangle ^{+}}\sum_{i=1}^{m(\lambda)}c_{\lambda,i}(g,n)\left(e^{-2\lambda(\log a)}\Ad^{-1}(a)X_{-\lambda,i} +e^{-\lambda(\log a)}(X_{\lambda,i}-X_{-\lambda,i})\right)\\ & +\sum_{\lambda \in \Sigma^+-\left\langle \Theta\right\rangle ^{+}}\sum_{i=1}^{m(\lambda)}c_{\lambda,i}(g,n)e^{-\lambda( \log a)}X_{\lambda,i}  \\&=\sum_{\lambda \in \left\langle \Theta\right\rangle ^{+}}\sum_{i=1}^{m(\lambda)}c_{\lambda,i}(g,n)e^{-2\lambda(\log a)}\Ad^{-1}(a)X_{-\lambda,i} \\&+\sum_{\lambda \in \left\langle \Theta\right\rangle ^{+}}\sum_{i=1}^{m(\lambda)}c_{\lambda,i}(g,n)e^{-\lambda(\log a)}(X_{\lambda,i}-X_{-\lambda,i})+\sum_{\lambda \in \Sigma^+-\left\langle \Theta\right\rangle ^{+}}\sum_{i=1}^{m(\lambda)}c_{\lambda,i}(g,n)e^{-\lambda( \log a)}X_{\lambda,i}.
   \end{align*}
Comparing this with the expression \eqref{eq:Ad2} we had obtained earlier for $\Ad^{-1}(a)N_1^+$, we obtain that 
\bq
\label{eq:A2}
A_2=0  ,
\eq
 and $N_2^-=\sum_{\lambda \in \left\langle \Theta\right\rangle ^{+}}\sum_{i=1}^{m(\lambda)}c_{\lambda,i}(g,n)e^{-2\lambda(\log a)}\Ad^{-1}(a)X_{-\lambda,i}$, since $\g=\k\oplus \a \oplus \n^-$, and $ \p_\Theta(K)\cap \a(\Theta)=\{0\}$. Therefore 
\begin{align}
\label{eq:8}
\begin{split}
 N_3^-&=N_1^-+\Ad(a)N_2^-\\=\sum_{\lambda \in \Sigma^+}\sum_{i=1}^{m(\lambda)}c_{-\lambda,i}(g,n)X_{-\lambda,i} &+ \sum_{\lambda \in \left\langle \Theta\right\rangle ^{+}}\sum_{i=1}^{m(\lambda)}c_{\lambda,i}(g,n)e^{-2\lambda(\log a)}X_{-\lambda,i},\\
  A_1+A_2& =\sum_{\alpha_i \in \Theta}c_i(g,n)H_i\quad \text{ mod} \, \a_\Theta.
\end{split} 
\end{align}
 As $N^-\times A(\Theta)$ can be identified with an open dense submanifold of the homogeneous space $G/P_\Theta(K)$, we have the isomorphisms $T_{gnaP_{\Theta}(K)}G/P_{\Theta}(K)\simeq T_{p}(N^{-}\times A(\Theta)) \simeq T_{n}N^{-}\bigoplus T_{a}(A(\Theta))$, where $p=(n,a)\in N^- \times A(\Theta)$. Therefore, by equation \eqref{eq:action} and the expressions for $N_3^- $ and  $A_1+A_2$,  we finally deduce that the fundamental vector field $Y_{|G/P_{\Theta}(K)}$ at a point $p$  corresponding to the action of $\exp(sY)$ on $G/P_\Theta(K)$ is given  by 
\begin{align*}
(Y_{|G/P_{\Theta}(K)})_{p}&=\sum_{\lambda\in \Sigma^+}\sum_{i=1}^{m(\lambda)} c_{-\lambda,i}(g,n)(X_{-\lambda,i})_p+ \sum_{\lambda\in \langle \Theta\rangle ^+} \sum_{i=1}^{m(\lambda)}c_{\lambda,i}(g,n)e^{-2\lambda \log a}(X_{-\lambda,i})_p \\ &+ \sum_{\alpha_i\in\Theta} c_{i}(g,n)(H_i)_p,
\end{align*} 
where $Y \in \g$, and  the coefficients are given by \eqref{eq:2211}.
\end{proof}

 Let us next state the following 
\begin{lemma}
\label{lemma:8}
Let $Y\in \n^{-}\oplus \a$ be given by  $Y=\sum_{\lambda\in \Sigma^+}\sum_{i=1}^{m(\lambda)}c_{-\lambda,i} X_{-\lambda,i} +\sum_{j=1}^{l}c_{j} H_j$,
and introduce the notation $t^\lambda=t_1^{\lambda(H_1)}\cdots t_l^{\lambda(H_l)}$. Then, via the identification of $N^-\times\R^l_+$ with $N^-A$ by $(n,t)\mapsto n\cdot exp(-\sum_{j=1}^{l}H_j\log t_j)$, the  left invariant vector field on the Lie group $N^-A$ corresponding to $Y$ is expressed as 
\bqn
\tilde Y_{|N^-\times \R^l_+}=\sum_{\lambda\in \Sigma^+}\sum_{i=1}^{m(\lambda)}c_{-\lambda,i}t^{\lambda}X_{-\lambda,i}-\sum_{j=1}^{l}c_{j}t_j\frac{\partial}{\partial t_j},
\eqn
and can analytically be extended to a vector field on $N^-\times\R^l$.
\end{lemma}
\begin{proof}
The lemma is  stated in   Oshima, \cite{oshima78}, Lemma 8, but for greater clarity, we include  a  proof of it here. Let $X_{-\lambda,i}$  be a fixed basis element of $\n^-$. The corresponding left-invariant vector field  on the Lie group $N^-A$ at the point $na $ is given by 
\begin{align*}
\frac{d}{ds}f(na\exp(sX_{-\lambda,i}))_{|s=0}&=\frac{d}{ds}f(n(a\exp(sX_{-\lambda,i})a^{-1})a)_{|s=0}= \frac{d}{ds}f(n\e{sAd(a)X_{-\lambda,i}}a)_{|s=0},
\end{align*}
where $f$ is a smooth function on $N^-A$. Regarded as a left invariant vector field on  $ N^ - \times \R^l_+$,  it is  therefore given by 
\bqn
\tilde X_{-\lambda,i|N^-\times \R^l_+}= \Ad(a)X_{-\lambda,i} =e^{-\lambda( \log a)}X_{-\lambda,i}=t^{\lambda}X_{-\lambda,i},
\eqn
compare \eqref{eq:7}.  Similarly, for  a basis element $H_i$ of $\a$ the corresponding left invariant vector field on $N^-A$ reads 
\begin{gather*}
\frac{d}{ds}f(na\exp(sH_i))_{|s=0}=\frac{d}{ds}f(n\exp(-\sum_{j=1}^{l}\log t_jH_j)\exp(sH_i))_{|s=0}\\=\frac{d}{ds}f\Big (n\exp(-\sum_{j=1}^{l}\log t_jH_j+sH_i)\Big )_{|s=0}=\frac{d}{ds}f\Big (n\exp(-\sum_{j\neq i}\log t_jH_j-\log (t_ie^{-s})H_i)\Big )_{|s=0},
\end{gather*}
and with the identification $N^- A \simeq N^ - \times \R^l_+$ we obtain
$$\tilde H_{i|N^-\times \R^l_+}=-t_i\frac{\partial}{\partial t_i}.$$ 
 As there are no negative powers of $t$,  $\tilde Y_{N^- \times \R^l_+}$ can be   extended analytically to $N^-\times\R^l$, and the lemma follows.
\end{proof}

Similarly,  by the identification $G/K \simeq N^-\times A\simeq N^- \times \R^l_+$ via the mappings $(n,t)\mapsto n\cdot exp(-\sum_{i=1}^{l}H_i\log t_i)\cdot a \mapsto gnaK$ one sees that 
the action  on $G/K$ of the fundamental vector field corresponding to 
$\exp (sY)$ , $Y \in \g$, is given by 
\begin{equation}
\label{eq:2.13}
Y_{|N^-\times \R^l_+}=\sum_{\lambda\in \Sigma^+}\sum_{i=1}^{m(\lambda)}(c_{\lambda,i}(g,n)t^{2\lambda}+ c_{-\lambda,i}(g,n))X_{-\lambda,i}-\sum_{i=1}^{l}c_{i}(g,n)t_i\frac{\partial}{\partial t_i},
\end{equation}
where the coefficients are given by \eqref{eq:2211}. 
 Again, the vector field \eqref{eq:2.13} can be   extended analytically to $N^-\times\R^l$, but  in contrast to the left invariant vector field $\tilde Y_{N^-\times \R^l}$, $Y_{N^-\times \R^l}$ does not necessarily vanish  if $t_1=\dots t_l=0$.  
We come now to the description of the Oshima compactification of the Riemannian symmetric space $G/K$. For this, let $\hat{\X}$ be the product manifold $G\times N^-\times \R^l$. Take $\hat{x}=(g,n,t)\in\hat{\X}$, where $g\in G,\,n\in N^-,\,t=(t_{1},\dots ,t_{l})\in\R^l$, and define an action of $G$ on $\hat{\X}$ by  $g'\cdot (g,n,t):=(g'g,n,t),\, g'\in G.$ For $s \in \R$, let 
\bqn
\sgn  s = \left \{\begin{array}{cl} s/ |s|, & s \not=0, \\ 0, & s=0,
\end{array}    \right.
\eqn
and put $\sgn \hat{x}=(\sgn t_1,\dots ,\sgn t_l)\in\{-1,0,1\}^l$. We then define the subsets $ \Theta_{\hat{x}}=\{\alpha_i\in\Delta: t_i\neq 0\}$. Similarly, let $ a(\hat{x})=\exp (-\sum_{t_{i}\neq 0} H_i\log|t_i|) \in A(\Theta_{\hat{x}})$. On $\hat{\X}$, define now an equivalence relation by setting 
\bqn
 \hat{x}=(g,n,t)\sim  \hat{x}'=(g',n,'t') \quad \Longleftrightarrow \quad \left \{
\begin{array}{l} 
a) \, \sgn \hat{x}=\sgn \hat{x}', \\
b) \,  g\, n\, a(\hat{x})\, P_{\Theta_{\hat{x}}}(K)=g'\, n'\, a(\hat{x}')\, P_{\Theta_{\hat{x}'}}(K).
\end{array} \right.
\eqn
Note that the condition  $\sgn \hat{x}=\sgn \hat{x}'$ implies that $\hat{x},\hat{x}' $ determine the same subset $\Theta_{\hat{x}}$ of $\Delta$, and consequently  the same group $P_{\Theta_{\hat{x}}}(K)$, as well as the same homogeneous space $G/P_{\Theta_{\hat{x}}}(K)$, so that condition $b)$ makes sense. It says that $gna(\hat{x}), \, g'n'a(\hat{x}') $ are in the same $P_{\Theta_{\hat{x}}}(K)$ orbit on $G$, corresponding to the right action  by $P_{\Theta_{\hat{x}}}(K)$ on $G$. We now define 
\bqn 
\widetilde{\X}:=\hat{\X}/\sim, 
\eqn
endowing it  with the quotient topology, and denote by $\pi:\hat{\X}\rightarrow \widetilde{\X}$  the canonical projection.  The action of $G $ on $\hat{\X}$ is compatible with the equivalence relation $\sim$, yielding a $G$-action   
$g'\cdot \pi (g,n,t):=\pi (g'g,n,t)$  on $\widetilde \X$.  For each $g\in G$, one can show that the maps
\bq
\label{eq:coord}
 \phi_g:N^-\times\R^l\rightarrow \widetilde{U}_g: (n,t)\mapsto \pi(g,n,t), \qquad \widetilde{U}_g=\pi (\{g\} \times N^- \times \R^l),  
\eq
are bijections. One has then the following

\begin{theorem} 
\label{Thm.1}
\begin{enumerate}
\item $\widetilde{\X}$ is a simply connected, compact, real-analytic manifold without boundary.

\item  $\widetilde{\X}=\cup_{w \in W} \widetilde{U}_{m_w}=\cup_{g\in G}\widetilde{U}_g $. For $ g\in G$,  $ \widetilde{U}_g $ is an open submanifold of $\widetilde{\X}$ topologized in such a way that the coordinate map $\phi_g$ defined above is a real-analytic diffeomorphism. Furthermore, $\widetilde{\X}\setminus \widetilde{U}_g$ is the union of a finite number of submanifolds of  $\widetilde{\X}$ whose codimensions in  $\widetilde{\X}$ are not lower than $2$. 

\item The action of $G$ on  $\widetilde{\X}$ is real-analytic. For a point $\hat{x}\in\hat{\X} $, the $G$-orbit of $\pi(\hat{x})$ is isomorphic to the homogeneous space $G/P_{\Theta_{\hat{x}}}(K)$, and for $\hat{x}, \hat{x}' \in \hat{\X}$ the $G$-orbits of $\pi(\hat{x})$ and $\pi(\hat{x}')$ coincide if and only if $\sgn \hat{x}=\sgn \hat{x}'$. Hence the orbital decomposition of $\widetilde{\X}$ with respect to the action of $G$ is of the form 
\bq
\label{eq:decomp}
\widetilde{\X}\simeq \bigsqcup_{\Theta\subset\Delta} 2^{\#\Theta}(G/P_\Theta(K)) \quad \text{(disjoint union)},
\eq
where $\#\Theta$ is the number of elements of $\Theta$  and $2^{\#\Theta}(G/P_\Theta(K))$ is the disjoint union of $2^{\#\Theta}$ copies of $G/P_\Theta(K)$. 
\end{enumerate}
\end{theorem}
\begin{proof}
See Oshima, \cite{oshima78}, Theorem 5.
\end{proof}

Next, for $\hat{x}=(g,n,t)$ define  the set $B_{\hat{x}}=\{ (t_1'\dots t_l')\in \R^l:\sgn t_i=\sgn t_i' ,1\leq i\leq l\}$. By analytic continuation, one can restrict the vector field  \eqref{eq:2.13} to $N^-\times B_{\hat{x}}$, and with the identifications $G/P_{\Theta_{\hat{x}}}(K) \simeq N^-\times A(\Theta_{\hat{x}})\simeq N^-\times B_{\hat{x}}$  via the maps
 $$ gnaP_{\Theta_{\hat{x}}}\leftarrow(n,a)\mapsto(n,\sgn t_1e^{-\alpha_1(\log a)},\dots, \sgn t_le^{-\alpha_l (\log a)}),$$
 one actually sees that this restriction coincides with  the vector field in  Lemma \ref{lemma:fundvec}. The action of the fundamental vector field on $\widetilde {\X}$ corresponding to $\exp{sY}, Y\in \g$,  is therefore given by the extension of  \eqref{eq:2.13} to $N^-\times \R^l$.
Note that for a simply connected nilpotent Lie group $N$ with Lie algebra $\n$, the exponential $\exp:\n\rightarrow N$ is a diffeomorphism. So, in our setting, we can identify $N^-$ with $\R^k$.  
Thus, for every point in $\widetilde{\X}$, there exists a local coordinate system $(n_1,\dots  , n_k,t_1,\dots , t_l)$ in a neighbourhood of that point such that two points $(n_1,\dots ,n_k,t_1,\dots ,t_l)$  and $(n'_1,\dots ,n'_k,t'_1,\dots ,t'_l)$ belong to the same $G$-orbit if, and only if,  $\sgn t_j=\sgn t'_j$, for $j=1,\dots, l$. This means that the orbital decomposition of $\widetilde{\X}$ is of \emph{normal crossing type}. In what follows, we shall identify the open $G$-orbit $\pi(\{\hat{x}=(e,n,t) \in\hat \X:\sgn \hat{x}=(1,\dots ,1)\})$ with the Riemannian
 symmetric space $G/K$,  and the orbit $\pi(\{\hat{x}\in\hat{\X}:\sgn \hat{x}=(0,\dots ,0)\}$ of lowest dimension with its  Martin boundary $G/P$.

\section{Review of pseudodifferential operators}
\label{sec:PDO}

{\bf{Generalities.}} This section is devoted to an exposition of some  basic facts about  pseudodifferential operators needed to formulate our main results in the sequel. For a detailed introduction to the field, the reader is referred to \cite{hoermanderIII} and \cite{shubin}.   Consider first an open set $U$ in $\R^n$, and let $x_1,\dots ,x_n$ be the standard coordinates. For any real number $l$,  we denote by $\Sym^l(U\times \R^n)$ the class of all functions $a(x,\xi)\in \Cinft(U\times \R^n)$ such that, for any multi-indices $\alpha,\beta$, and any compact set $\mathcal{K}\subset U$, there exist constants $C_{\alpha,\beta,\mathcal{K}}$ for which
\begin{equation}
  \label{H}
|(\gd ^\alpha_\xi\gd ^\beta_x a)(x,\xi)| \leq C_{\alpha,\beta,\mathcal{K}} \eklm{\xi}^{l-|\alpha|}, \qquad x \in \mathcal{K},  \quad \xi \in \R^n,
\end{equation}
where $\eklm{\xi}$ stands for $(1+|\xi|^2)^{1/2}$, and $|\alpha|=\alpha_1+\dots +\alpha_n$. We further put $\Sym^{-\infty}(U\times \R^n)=\bigcap _{l \in \R} \Sym^l(U\times \R^n)$. Note that, in general, the constants $C_{\alpha,\beta,K}$ also depend on $a(x,\xi)$. For any such $a(x,\xi)$ one then defines the continuous linear operator
\begin{displaymath}
  A:\CT(U) \longrightarrow \Cinft(U)
\end{displaymath}\large\Large\normalsize
by the formula
\begin{equation}
  \label{I}
Au(x)=\int e^{ix \cdot \xi} a(x,\xi) \hat u(\xi) \dbar \xi,
\end{equation}
where $\hat u$ denotes the Fourier transform of $u$, and $\dbar \xi=(2\pi)^{-n} \d \xi$. \footnote{Here and in what follows we  use the convention that, if not specified otherwise, integration is to be performed 
over Euclidean space.} An operator $A$ of this form is called a \emph{pseudodifferential operator of order l}, and we denote the class of all such operators for which $a(x,\xi) \in \Sym^l(U\times \R^n)$  by $\L^l(U)$. The set $\L^{-\infty}(U)=\bigcap_{l\in \R} \L^l (U)$ consists of all operators with smooth kernel. They are called \emph{smooth operators}. By inserting  in \eqref{I} the definition of $\hat u$, we obtain for $Au$ the expression
\begin{equation}
  \label{II}
Au(x)=\int \int  e^{i(x-y) \cdot \xi} a(x,\xi)  u(y) \d y \,\dbar \xi, 
\end{equation}
which has a suitable regularization as an oscillatory integral. The Schwartz kernel of $A$ is a distribution $K_A\in \D '(U\times U)$ which  is given the oscillatory integral
\begin{equation}
  \label{III}
K_A(x,y)=\int e^{i(x-y)\cdot \xi} a(x,\xi) \,\dbar \xi.
\end{equation}
It is a smooth function off the diagonal in $U\times U$.
Consider next a $n$-dimensional paracompact $\Cinft$ manifold ${\bf X}$, and let  $\{(\kappa_\gamma, \widetilde U^\gamma)\}$ be an atlas for ${\bf X}$. Then a linear operator
\begin{equation}
\label{IIIa}
  A:\CT({\bf X}) \longrightarrow \Cinft({\bf X})
\end{equation}
is called a \emph{pseudodifferential operator on ${\bf X}$ of order $l$} if for each chart diffeomorphism $\kappa_\gamma:\widetilde U^\gamma \rightarrow U^\gamma= \kappa_\gamma(\widetilde U^\gamma)$,  the operator $A^{\gamma} u = [A_{|\widetilde U^\gamma} ( u\circ \kappa_{\gamma})] \circ \kappa_{\gamma}^{-1}$ given by the diagram
\begin{displaymath}
\begin{CD} 
\CT(\widetilde U^{\gamma})       @>{A_{|\widetilde U^\gamma}}>>   \Cinft(\widetilde U^\gamma)               \\
@A {\kappa_{\gamma}^\ast}AA @AA {\kappa_{\gamma}^\ast}A\\
 \CT( U^{\gamma})  @> {A^{\gamma}}>>  \Cinft( U^\gamma)       
\end{CD}
\end{displaymath}
is a pseudodifferential operator on $U^\gamma$ of order $l$, and its kernel $K_A$ is smooth off the diagonal. In this case we write $A \in \L^l({\bf X})$. Note that, since the $\widetilde U^\gamma$ are not necessarily connected, we can choose them in such a way that ${\bf X}\times {\bf X}$ is covered by the open sets $\widetilde U^\gamma \times \widetilde U^\gamma$. Therefore the condition that $K_A$ is smooth off the diagonal can be dropped. 
Now, in general, if ${\bf X}$ and ${\bf Y}$ are two smooth manifolds, and 
\begin{equation*}
  A: \CT({\bf X}) \longrightarrow \Cinft({\bf Y}) \subset \D'({\bf Y})
\end{equation*}
is a continuous linear operator, where $\D'({\bf Y})=(\CT({\bf Y},\Omega))'$ and $\Omega=|\Lambda^n({\bf Y})|$ is the density bundle on ${\bf Y}$, its Schwartz kernel is given by the distribution section $K_A \in \D'({\bf Y} \times {\bf X}, {\bf 1} \boxtimes \Omega_{{\bf X}})$, where  $\D'({\bf Y}\times {\bf X} ,1 \boxtimes \Omega_{{\bf X}}) = (\CT({\bf Y} \times {\bf X}, ({{\bf 1}} \boxtimes \Omega_{{\bf X}})^\ast \otimes \Omega_{{\bf Y}\times {\bf X}}))'$. Observe that  $\CT ({\bf Y}, \Omega_{{\bf Y}}) \otimes \Cinft ({\bf X}) \simeq \Cinft ( {\bf Y} \times {\bf X}, ({\bf 1} \boxtimes \Omega_{{\bf X}})^\ast \otimes\Omega_{{\bf Y} \times {\bf X}})$. 
In case that  ${\bf X}={\bf Y}$  and $A\in \L^l({\bf X})$, $A$ is given locally by the operators $A^{\gamma}$, which can be written in the form 
  \begin{equation*}
    A^{\gamma}u(x) = \int \int e^{i(x-y) \cdot \xi} a ^{\gamma}(x,\xi) u(y) \d y \dbar \xi,
  \end{equation*}
where $u \in \CT(U^{\gamma})$, $x \in U^\gamma$, and $a^{\gamma}(x,\xi) \in {\rm{S}}^l(U^\gamma, \R^n)$. The kernel of $A$ is then determined by the kernels $K_{A^{\gamma}} \in \D'(U^\gamma \times U ^{\gamma})$. For $l < -\dim {\bf X}$, they are continuous, and given by absolutely convergent integrals. 
In this case, their restrictions to the respective diagonals in $U^\gamma \times U^\gamma$ define continuous functions
\begin{equation*}
  k^\gamma(m)= K_{A^{\gamma}} (\kappa_\gamma (m),\kappa _\gamma (m)), \qquad m \in \widetilde U^\gamma,
\end{equation*}
which, for $m \in  \widetilde U^{\gamma_1} \cap \widetilde U^{\gamma_2}$, satisfy the relations $ k^{\gamma_2}(m) =| \det (\kappa_{\gamma_1} \circ \kappa ^{-1}_{\gamma_2})' | \circ \kappa_{\gamma_2}(m) k^{\gamma_1}(m)$, and therefore define a density $k \in C ({\bf X},\Omega)$ on $\Delta_{{\bf X}}\times {\bf X} \simeq {\bf X}$.   
If ${\bf X}$ is compact, this density can be integrated, yielding the trace of the operator $A$, 
\bq
\label{eq:trace}
\tr A=\int _{{\bf X}} k=\sum_\gamma \int_{U^\gamma} (\alpha_\gamma \circ \kappa_\gamma^{-1}) (x) \, K_{A^\gamma}(x,x)  \d x,
\eq
where $\{\alpha_\gamma\}$ denotes a partition of unity subordinated to the atlas $\{(\kappa_\gamma, \widetilde U^\gamma)\}$, and $dx$ is Lebesgue measure in $\R^n$. 

{\bf{Totally characteristic pseudodifferential operators.}} We introduce now a special class of pseudodifferential operators associated in a natural way to a $\Cinft$ manifold ${\bf X}$ with boundary $\gd {\bf X}$.  Our main reference 
will be \cite{melrose} in this case. Let $\Cinft({\bf X})$ be the space of functions on ${\bf X}$ which are $\Cinft$ up to the boundary, and $\Cinftv({\bf X})$ the subspace of functions vanishing to all orders on $\gd {\bf X}$. The standard spaces of distributions over ${\bf X}$ are 
\begin{equation*}
  \D'({\bf X})= (\CTv({\bf X},\Omega))', \qquad \dot \D({\bf X})' =(\CT({\bf X},\Omega))', 
\end{equation*}
the first being the space of \emph{extendible distributions}, whereas the second is the space of \emph{distributions supported by ${\bf X}$}. Consider now the translated partial Fourier transform of a symbol $a(x,\xi) \in \Sym^l(\R^n\times \R^n)$,
\begin{equation*}
  Ma(x,\xi';t)=\int e^{i(1-t)\xi_1} a(x,\xi_1,\xi') d\xi_1,
\end{equation*}
where we wrote $\xi=(\xi_1,\xi')$. $Ma(x,\xi';t)$ is $\Cinft$ away from $t=1$, and one says that $a(x,\xi)$ is \emph{lacunary} if it satisfies the  condition\begin{equation}
\label{V}
  Ma(x,\xi';t)=0 \qquad \text{ for } t<0.
\end{equation}
The subspace of lacunary symbols will be denoted by $\Sym^l_{la}(\R^n\times \R^n)$. Let $Z=\overline{\R^+}  \times \R^{n-1}$ be the standard manifold with boundary  with the natural coordinates $x=(x_1,x')$. In order to define on $Z$ operators of the form \eqref{II}, 
where now $a (x,\xi)=\widetilde a(x_1,x',x_1\xi_1, \xi')$ is a more general amplitude and $\widetilde a(x,\xi)$ is lacunary, one rewrites the formal adjoint of $A$ by making a singular coordinate change. Thus, for $u \in \CT(Z)$, one considers 
\begin{equation*}
  A^\ast u(y) =\int\int e^{i(y-x)\cdot \xi} \overline{a}(x,\xi) u(x) \, \d x \dbar \xi.
\end{equation*}
By putting $\lambda=x_1\xi_1$, $s=x_1/y_1$, this can be rewritten as
\begin{equation}
\label{VII}
  A^\ast u(y)=(2\pi)^{-n}\int \int \int \int e^{i(1/s-1,y'-x')\cdot (\lambda,\xi')}\overline {\widetilde a}(y_1 s, x',\lambda,\xi') u(y_1 s,x')d\lambda \frac{ds}s dx' d\xi'.
\end{equation}
According to  \cite{melrose}, Propositions 3.6 and 3.9, for every $\widetilde a \in \Sym_{la}^{-\infty}(Z\times \R^n) $, the successive integrals in \eqref{VII} converge absolutely and uniformly, thus defining  a continuous bilinear form
  \begin{equation*}
    \Sym^{-\infty}_{la}(Z\times \R^n) \times \CT(Z) \longrightarrow \Cinft(Z),
  \end{equation*}
which extends to a separately continuous form
\begin{equation*}
  \Sym^\infty_{la}(Z\times \R^n) \times \CT(Z) \longrightarrow \Cinft(Z).
\end{equation*}
If $\widetilde  a \in \Sym^\infty_{la}(Z \times \R^n)$ and $a(x,\xi)=\widetilde a(x_1,x',x_1\xi_1,\xi')$, one then defines the operator 
\begin{equation}
\label{VIII}
  A:\dot \E'(Z) \longrightarrow \dot \D'(Z), 
\end{equation}
written formally as \eqref{II}, as the adjoint of $A^\ast$. In this way, the oscillatory integral \eqref{II} is identified with a separately continuous bilinear mapping
\begin{equation*}
  \Sym^\infty_{la}(Z\times \R^n) \times \dot \E'(Z) \longrightarrow \dot \D'(Z).
\end{equation*}
 The space $\L^l_b(Z)$ of \emph{totally characteristic pseudodifferential operators on $Z$ of order $l$} consists of those continuous linear maps \eqref{VIII} such that for any $u,v \in \CT(Z)$, $v Au$ is of the form \eqref{II} with $a(x,\xi)=\widetilde{a}(x_1,x',x_1\xi_1,\xi ')$ and $\widetilde a(x,\xi)\in \Sym^l_{la}(Z\times \R^n)$. Similarly, 
a continuous linear map \eqref{IIIa} on a smooth manifold $\bf X$ with boundary $\gd \bf X$ is said to be an element of the space $\L^l_{b}({\bf X})$ of \emph{totally characteristic pseudodifferential operators on ${\bf X} $ of order $l$}, if for a given atlas $(\kappa_\gamma,\widetilde U^\gamma)$ the operators $A^\gamma u=[A_{|\widetilde U^\gamma} (u \circ \kappa_\gamma)] \circ \kappa^{-1}_\gamma$ are elements of $\L^l_b(Z)$,  
where the $\widetilde U^\gamma$ are coordinate patches isomorphic to subsets in $Z$.

In an analogous way, it is possible to introduce the concept of a totally characteristic pseudodifferential operator on a manifold with corners. As the standard manifold with corners, consider
\bqn 
\R^{n,k}=[0,\infty)^k \times \R^{n-k}, \qquad 0 \leq k \leq n,
\eqn 
with coordinates $x=(x_1,\dots, x_k, x')$. A  \emph{totally characteristic pseudodifferential operator on $\R^{n,k}$ of order $l$} is locally given by an oscillatory integral \eqref{II} with $a(x,\xi) =\widetilde a (x, x_1 \xi_1,\dots, x_k \xi_k, \xi')$, where now $\widetilde a(x,\xi)$ is a symbol of order $l$ that satisfies the lacunary condition for each of the coordinates $x_1,\dots, x_k$, i.e.
\bqn 
\int e^{i(1-t)\xi_j} a(x, \xi) \d \xi_j =0 \qquad \text{for } t<0 \text{ and } 1\leq j \leq k.
\eqn
In this case we write $\widetilde a(x,\xi) \in \Sym^l_{la}(\R^{n,k}\times \R^n)$.
A continuous linear map \eqref{IIIa} on a smooth manifold $\bf X$ with  corners is then said to be an element of the space $\L^l_{b}({\bf X})$ of \emph{totally characteristic pseudodifferential operators on ${\bf X} $ of order $l$}, if for a given atlas $(\kappa_\gamma,\widetilde U^\gamma)$ the operators $A^\gamma u=[A_{|\widetilde U^\gamma} (u \circ \kappa_\gamma)] \circ \kappa^{-1}_\gamma$ are   totally characteristic pseudodifferential operator on $\R^{n,k}$ of order $l$,  where the $\widetilde U^\gamma$ are coordinate patches isomorphic to subsets in $\R^{n,k}$. For an extensive treatment,  we refer the reader to \cite{loya98}.

\section{Invariant integral operators}
\label{Sec:4}

Let $\widetilde \X$ be the Oshima compactification of a Riemannian symmetric space $\X\simeq G/K$ of non-compact type. As was already explained, $G$ acts analytically on $\widetilde \X$, and the orbital decomposition is of normal crossing type. Consider  the Banach space $\mathrm{C}(\widetilde \X)$ of continuous, complex valued functions on $\widetilde \X$, equipped with the supremum norm, and let  $(\pi,\mathrm{C}(\widetilde \X))$ be the corresponding continuous  regular representation of $G$ given by
\bqn 
\pi(g) \phi(\tilde x) =\phi(g^{-1} \cdot \tilde x), \qquad \phi \in \mathrm{C}(\widetilde \X).
\eqn
  The representation of the universal enveloping algebra $\U$ of the complexification $\g_\C$ of $\g$ on the space of differentiable vectors $\mathrm{C}(\widetilde \X)_\infty$ will be denoted by $d\pi$. We will also consider the regular representation  of $G$ on $\Cinft(\widetilde \X)$ which, equipped with the topology of uniform convergence on compact subsets, becomes a Fr\'{e}chet space. This representation will be denoted by $\pi$ as well. Let $(L,\Cinft (G))$ be the left regular representation of $G$. With respect to the left-invariant metric on $G$ given by $\langle,\rangle_{\theta}$, we define  $d(g,h)$ as  the distance between two points $g,h \in G$, and set $|g|=d(g,e)$, where $e$ is the identity element of $G$.  A function $f$ on $G$ is \emph{at most of exponential growth}, if there exists a $\kappa>0$ such that $|f(g)| \leq C e^{\kappa|g|}$ for some constant $C>0$.
As before,  denote  a Haar measure on $G$ by $d_{G}$. Consider next the space   $\S(G)$ of rapidly decreasing functions on $G$ introduced in   \cite{ramacher06}. 
\begin{definition}
The space of rapidly decreasing functions on $G$, denoted by $\S(G)$, is given by all functions $f \in \Cinft(G)$ satisfying the following conditions:
\begin{itemize}
\item[i)] For every $\kappa \geq 0$, and $X \in \U$, there exists a constant C such that 
	$$|dL(X)f(g)| \leq C e^{-\kappa |g|} ;$$ 
\item[ii)] for every $\kappa \geq 0$, and $X \in \U$, one has $dL(X)f \in \L^1(G,e^{\kappa|g|}d_G)$.
\end{itemize}
\end{definition}
For later purposes, let us recall  the following integration formulas.
\begin{proposition}
\label{prop:A}
Let $f_1\in\S(G)$, and assume that $f_2 \in \Cinft(G)$, together with all its derivatives, is at most of exponential growth.  Let $X_1, \dots, X_d$ be a basis of $\g$, and  for  $X^\gamma=X^{\gamma_1}_{i_1}\dots X^{\gamma_r}_{i_r}$ write $X^{\tilde \gamma}=X^{\gamma_r}_{i_r}\dots X^{\gamma_1}_{i_1}$, where $\gamma$ is an  arbitrary multi-index.  Then
\begin{equation*}
\int_{G}f_1(g) dL(X^\gamma) f_2(g) d_{G}(g)=(-1)^{|\gamma|} \int _{G}
dL(X^{\tilde \gamma}) f_1(g) f_2(g) d_{G}(g).
\end{equation*}
\end{proposition}
\begin{proof}
See \cite{ramacher06}, Proposition 1.
\end{proof}
Next, we associate to every $f\in \S(G)$ and $\phi \in \mathrm{C}(\widetilde \X)$ the element $\int _{G} f(g)  
\pi(g) \phi \d_{G}(g)\in \mathrm{C}(\widetilde \X)$. It  is defined as a Bochner integral,  and the  continuous linear operator  on $\mathrm{C}(\widetilde \X )$ obtained this way is denoted by \eqref{eq:1}. Its restriction to $\Cinft(\widetilde \X)$ induces a continuous linear operator            
\begin{equation*}
\pi(f):\Cinft(\widetilde \X) \longrightarrow \Cinft(\widetilde \X) \subset \D'(\widetilde \X),
\end{equation*}
with Schwartz kernel given by the distribution section  $\mathcal{K}_f \in  \D'(\widetilde \X \times \widetilde \X, {{\bf 1}} \boxtimes \Omega_{\widetilde \X})$. The  properties of the Schwartz kernel $\mathcal{K}_f$ will depend   on the analytic properties of $f$, as well as the orbit structure of the underlying $G$-action, and our main effort  will be directed towards the elucidation of the structure of $\mathcal{K}_f$. For this, let us  consider the orbital decomposition \eqref{eq:decomp} of $\widetilde \X$, and remark that the restriction of $\pi(f) \phi$ to any of the connected components isomorphic to $G/P_\Theta(K)$ depends only on the restriction of $\phi\in \mathrm{C}(\widetilde \X)$ to that component, so that we obtain the continuous linear operators
\begin{equation*}
\pi(f)_{|\widetilde \X_\Theta}:\CT(\widetilde \X_\Theta) \longrightarrow \Cinft(\widetilde \X_\Theta),
\end{equation*}
where $\widetilde \X_\Theta$ denotes a component in $\widetilde \X$ isomorphic to $G/P_\Theta(K)$. Let us now assume that $\Theta=\Delta$, so that  $P_\Theta(K)=K$. Since $G$ acts transitively on $\widetilde \X_{\Delta}$ one deduces that $\pi(f)_{|\widetilde \X_\Delta} \in \L^{-\infty}(\widetilde \X_{\Delta})$, c.p.   \cite{ramacher06},  Section 4.  The main goal of this section is to prove  that the restrictions of the operators $\pi(f)$ to the manifolds with corners $\overline{\widetilde \X_{\Delta}}$  are  totally characteristic pseudodifferential operators of class $\L^{-\infty}_b$.

Let $\mklm{(\widetilde U_{m_w}, \phi_{m_w}^{-1})}_{w\in W}$ be the finite  atlas on the Oshima compactification $\widetilde \X$ defined earlier. For each $\tilde x \in \widetilde \X$, let $\widetilde W_{\tilde x}$ be an open neighborhood of $\tilde x$ contained in some $\widetilde U_{m_w}$  such that $\mklm{h \in G: h \widetilde W_{\tilde x} \subset \widetilde {U}_{m_w}}$ 
acts transitively on the $G$-orbits of $\widetilde{W}_{\tilde x}$, c.p. \cite{ramacher06}, Section 6. We obtain a finite atlas $\mklm{(\widetilde{W}_\gamma, \phi^{-1}_{m_{w_\gamma}})}_{\gamma \in I}$ of $\widetilde \X$ satisfying the following properties:
\begin{itemize}
\item[i)] For each $\widetilde{W}_\gamma$, there exist open sets $V_\gamma \subset V_\gamma^1 \subset G$, stable under inverse, that act transitively on the $G$-orbits of $\widetilde{W}_\gamma$;
\item[ii)] For all $\gamma \in I$ one has $V_\gamma^1 \cdot \widetilde{W}_\gamma \subset \widetilde{U}_{m_{w_\gamma}}$ for some $m_{w_\gamma}\in M^*$.
\end{itemize}
To simplify notation, we shall write  $\phi_\gamma$ instead of $ \phi_{m_{w_\gamma}}$. 
 Consider now the localization of the operators $\pi(f)$ with respect to the finite atlas $\mklm{(\widetilde{W}_\gamma, \phi^{-1}_\gamma)}_{\gamma \in I}$  given by 
 \bqn 
 A_{f}^{\gamma} u=[\pi(f)_{|\widetilde W_{\gamma}}(u\circ \phi_{\gamma}^{-1})]\circ \phi_{\gamma},  \qquad u\in\CT(W_{{\gamma}}), \, W_{\gamma}=\phi^{-1}_{\gamma}(\widetilde W_\gamma),
 \eqn 
  see  Section \ref{sec:PDO}. Writing $\phi_{\gamma}^{g}= \phi_{\gamma}^{-1}\circ g^{-1}\circ \phi_{\gamma}$ and $x=(n,t)\in {W}_\gamma$ we obtain
\begin{equation*}
A_{f}^{\gamma} u(x)=\int_{G}f(g)\pi(g)(u\circ \phi_{\gamma}^{-1})(\phi_{\gamma}(x))dg=\int_{G}f(g)(u\circ \phi_{\gamma}^{g})(x)dg.
\end{equation*}
 Since we can restrict the domain of integration to $V_\gamma$, the latter integral can be rewritten as
\begin{equation*}
A_{f}^{\gamma} u(x)=\int_{G}c_{\gamma}(g) f(g)(u\circ \phi_{\gamma}^{g})(x)dg,
\end{equation*}
where $c_{\gamma}$ is a smooth bounded function on $G$ with support in $V_\gamma^1$ such that $c_{\gamma}\equiv 1$ on $V_\gamma$. Define next
\begin{equation}
\hat{f}_\gamma(x,\xi)=\int_{G}e^{i\phi_{\gamma}^{g}(x)\cdot\xi} c_{\gamma}(g)f(g) dg, \qquad a_{f}^{\gamma}(x,\xi)=e^{-ix\cdot \xi}\hat{f}_\gamma(x,\xi).
\end{equation}
Differentiating under the integral we see that $\hat{f}_\gamma(x,\xi),a_{f}^{\gamma}(x,\xi) \in\Cinft(W_{{\gamma}} \times \R^{k+l})$.
Let us next state the following
\begin{lemma}
\label{lemma:expansion}
For any $\tilde x=\phi_{\gamma} (n,t) \in \widetilde{W}_\gamma$ and $g\in V_\gamma^1$ we have the power series expansion
\bq
\label{eq:expansion}
t_j(g\cdot \tilde x)=\displaystyle\sum_{\substack{\alpha,\beta \\ \beta_{j} \neq 0}} c^{j}_{\alpha,\beta}(g)n^{\alpha}(\tilde x)t^{\beta}(\tilde x), \qquad  j=1,\dots, l,
\eq
where the coefficients $c^{j}_{\alpha,\beta}(g)$ depend real-analytically on $g$, and $\alpha, \beta$ are multi-indices.  
\end{lemma}
\begin{proof}
By Theorem \ref{Thm.1}, a $G$-orbit in  $\widetilde \X$ is locally determined by the signature of any of its elements. In particular, for $\tilde x \in \widetilde W_\gamma$, $ g\in V_\gamma^1$ we have $\sgn{t_j(g\cdot \tilde x)}=\sgn{t_j(\tilde x)}$ for all $j=1,\dots, l$. Hence, $t_j(g\cdot \tilde x)=0$ if and only if $t_j(\tilde x)=0$. Now, due to the analyticity of the coordinates $(\phi_{\gamma}, \widetilde W_\gamma)$,  there is a power series expansion 
\begin{equation*}
t_j(g\cdot \tilde x)=\displaystyle\sum_{\alpha,\beta } c^{j}_{\alpha,\beta}(g)n^{\alpha}(\tilde x)t^{\beta}(\tilde x), \qquad \tilde x\in \widetilde{W}_\gamma, \,  g\in V_\gamma^1,
\end{equation*} 
for every  $j=1,\dots,l$, which can be rewritten as
\begin{equation}\label{eqn.10}
t_j(g\cdot \tilde x)=\displaystyle\sum_{\substack{\alpha,\beta \\ \beta_{j} \neq 0}} c^{j}_{\alpha,\beta}(g)n^{\alpha}(\tilde x)t^{\beta}(\tilde x) + \displaystyle\sum_{\substack{\alpha,\beta \\ \beta_{j}= 0}} c^{j}_{\alpha,\beta}(g)n^{\alpha}(\tilde x)t^{\beta}(\tilde x).
\end{equation}
Suppose $t_j(\tilde x)=0$. Then the  first summand of the last equation must vanish, as in each term of the summation a non-zero power of $t_j(\tilde x)$ occurs. Also, $t_j(g\cdot \tilde x) =0$. Therefore \eqref{eqn.10} implies that the second summand must vanish, too.  But  the latter is independent of $t_j$. So we conclude  
\begin{equation*}
 \displaystyle\sum_{\substack{\alpha,\beta \\ \beta_{j}= 0}} c^{j}_{\alpha,\beta}(g)n^{\alpha}(\tilde x)t^{\beta}(\tilde x)\equiv{0}
 \end{equation*}
for all  $\tilde x \in \widetilde{W}_\gamma$, $ g\in V_\gamma^1$, and the assertion follows. 
\end{proof}
From Lemma \ref{lemma:expansion} we deduce that 
\bq
\label{eq:powers of t}
t_j(g\cdot \tilde x)=t_j^{q_j}(\tilde x) \chi_j(g,\tilde x), \qquad \tilde x \in \widetilde{W}_\gamma, \, g \in V_\gamma^1,
\eq
 where $\chi_j(g,\tilde x)$ is a function that is real-analytic in $g$ and in $\tilde x$, and $q_j$ is the lowest power of $t_j$ that occurs in the expansion \eqref{eq:expansion}, so that
 \bq
 \label{eq:notzero}
 \chi_j(g,\tilde x) \not= 0 \qquad \forall \, \tilde x \in \widetilde W_\gamma, \, g \in V_\gamma^1.
 \eq
 Indeed, $\chi_j(g,\tilde x)$ can only vanish if $t_j(\tilde x)=0$. But if this were the case, $q_j$ would not be the lowest power, and we obtain \eqref{eq:powers of t}. Furthermore, since $t_j(g \cdot \tilde x) = t_j(\tilde x)$ for $g=e$, one has $q_1=\dots= q_l$. 
  Thus, for $\tilde x=\phi_{\gamma} (x) \in \widetilde{W}_\gamma$, $x=(n,t)$, $g \in V_\gamma^1$, we have
\begin{equation*}
\phi^{g}_{\gamma} (x)=(n_1(g\cdot \tilde x),\dots,n_k(g\cdot \tilde x),t_1(\tilde x)\chi_{{}_1}(g,\tilde x),\dots,t_l(\tilde x)\chi_{{}_l}(g,\tilde x)).
\end{equation*}
Note that similar formulas hold for $\tilde x \in \widetilde U_{m_w}$ and $g$ sufficiently close to the identity. The following lemma  describes the $G$-action on $\widetilde \X$ as far as the $t$-coordinates are concerned.

\begin{lemma}
\label{lemma:char}
Let $X_{-\lambda,i}$ and  $H_j$ the basis elements for $\n^-$ and $\a$ introduced in Section \ref{Sec:2},  $w \in W$, and  $\tilde x \in  \widetilde U_{m_{w}}$. Then, for small $s\in \R$,
\bqn
 \chi_{{}_{j}}(\e{sH_i},\tilde x)=e^{-c_{ij}(m_{w})s},
 \eqn
  where the $c_{ij}(m_{w})$ represent the matrix coefficients of the adjoint representation of $M^\ast$ on $\a$, and  are  given by  $\Ad(m_{w}^{-1})H_i=\sum_{j=1}^lc_{ij}(m_{w})H_j$.
Furthermore, when $\tilde x =\pi(e,n,t)$,
 $$\chi_{{}_j}(\e{sX_{-\lambda,i}},\tilde x)\equiv 1 .$$ 
\end{lemma}
\begin{proof}
Let $Y \in \g$. As we saw in  the proof of Lemma \ref{lemma:fundvec},  the action of the one-parameter group $\exp(sY)$ on the  homogeneous space $G/P_{\Theta}(K)$ is given by  equation
 \eqref{eq:action}, where $N_3^-(s) \in\n^-,A_1(s)\in\a, A_2(s) \in\a(\Theta)$. Denote  the derivatives of $N_3^-(s)$, $A_1(s)$, and $A_2(s)$ at $s=0$ by $N_3^-$, $ A_1$, and $A_2$ respectively. The  analyticity of the $G$-action  implies  that  $N_3^-(s),A_1(s),A_2(s)$ are real-analytic functions in $s$. Furthermore, from \eqref{eq:action} it is clear that $N_3^-(0)=0$, $A_1(0)+A_2(0)=0$, so that for small $s$ we have 
\begin{align*}
A_1(s)+A_2(s)&=( A_1+A_2)\, s+ \frac{1}{2}\frac{d^2}{d s^{2}}(A_1(s)+A_2(s))|_{s=0} \, s^2+\dots \\
N_3^-(s)&=N_3^- \, s+\frac{1}{2}\frac{d^2}{d s^{2}}N_3^-(s)|_{s=0} \, s^2+\dots. 
\end{align*}
Next, fix $m_{w} \in M^*$ and let $\Theta = \Delta$. The action of the one-parameter group corresponding to $H_i$  at $\tilde x =\pi(m_{w},n,t) \in  \widetilde U_{m_{w}} \cap \widetilde \X_\Delta $ is given by
\begin{align*}
\exp(sH_i)m_{w}naK&=
m_{w}\left(m_{w}^{-1}\exp(sH_i)m_{w}\right)naK= m_{w} \exp(s\Ad(m_{w}^{-1})H_i)naK.
\end{align*}
As $m_{w}$ lies in $M^*$, $\exp(s\Ad(m_{w}^{-1})H_i)$ lies in $A$. Since $A$ normalizes $N^-$, we conclude  that $\exp(s\Ad(m_{w}^{-1})H_i)n\exp(-s\Ad(m_{w}^{-1})H_i)$ belongs to $N^-$. Writing $$n^{-1}\exp(s\Ad(m_{w}^{-1})H_i)n\exp(-s\Ad(m_{w}^{-1})H_i)=\exp N_3^-(s)$$ 
 we get 
 $$\exp(sH_i)m_{w}naK=m_{w}n\exp N_3^-(s)a\exp(s\Ad(m_{w}^{-1})H_i)K.$$
 In the notation of \eqref{eq:action} we therefore obtain  $A_1(s)+A_2(s)=s\Ad(m_{w}^{-1})H_i$, and by  writing $\Ad(m_{w}^{-1})H_i=\sum_{j=1}^lc_{ij}(m_{w})H_j$ we arrive at
 \begin{align*}
 a\exp(A_1(s)+A_2(s))&=\exp\Big (\sum_{j=1}^l(c_{ij}(m_{w})s-\log t_j)H_j\Big ).
 \end{align*}
 In terms of the coordinates this shows that $t_j(\exp(sH_i)\cdot \tilde x)=t_j(\tilde x)e ^{-c_{ij}(m_{w})s}$ for $\tilde x \in \widetilde U_{m_w} \cap \widetilde \X_\Delta$, and by analyticity we obtain that  $\chi_{{}_{j}}(\e{sH_i},\tilde x)=e^{-c_{ij}(m_{w})s}$ for arbitrary $\tilde x \in \widetilde U_{m_w}$. 
 On the other hand,  let $Y=X_{-\lambda,i}$, and  $\tilde x =\phi_e(n,t)\in \widetilde U_e\cap \widetilde \X_\Delta$. Then the action corresponding to  $X_{-\lambda,i}$  at  $\tilde x $ is given by 
  \begin{align*}
 \exp(sX_{-\lambda,i})naK=n\exp N_3^-(s)aK,
  \end{align*}
  where we wrote $\exp N_3^-(s)=s \Ad(n^{-1})\exp X_{-\lambda,i}$.   In terms of the coordinates this implies that  $t_j(\exp(sX_{-\lambda,i})\cdot \tilde x)=t_j(\tilde x)$ showing that $\chi_{{}_j}(\e{sX_{-\lambda,i}},\tilde x)\equiv 1$ for $\tilde x \in \widetilde U_e \cap \widetilde \X_\Delta$, and,  by  analyticity, for general $\tilde x \in \widetilde U_e$, finishing the proof of the lemma.
\end{proof}

Let now $x=(n,t)\in W_\gamma$, and define the matrix
\begin{equation} 
T_x = \begin{pmatrix} t_1& & 0\\& \ddots &\\0& & t_l\end{pmatrix}, 
\end{equation}
 so that  for  $\tilde x=\phi_{\gamma} (x) \in \widetilde{W}_\gamma$, $\, g \in V_\gamma^1$,
\begin{equation*}
({\bf 1}_k\otimes T^{-1}_{x})(\phi^{g}_{\gamma} (x))=(x_1(g\cdot \tilde x),\dots,x_k(g\cdot \tilde x),\chi_{{}_1}(g,\tilde x),\dots,\chi_{{}_l}(g,\tilde x)),
\end{equation*}
and set
\bqn
\psi^\gamma_{\xi,x}(g)=e^{i({\bf 1}_k\otimes T^{-1}_{x})(\phi^{g}_{\gamma} (x))\cdot\xi},
\eqn
where $\xi=(\xi_1,\dots,\xi_{k+l})\in\R^{k+l}$.  Also, introduce the auxiliary symbol 
\bq
\label{eq:auxsym}
\tilde{a}_f^{\gamma}(x,\xi)=a_{f}^{\gamma}(x,({\bf 1}_k\otimes T^{-1}_{x})\xi)=e^{-i(x_1,\dots,x_k,1,\dots,1).\xi}\int_{G}\psi^\gamma_{\xi,x}(g) c_{\gamma}(g) f(g) dg.
 \eq
Clearly, $\tilde{a}_f^{\gamma}(x,\xi)\in\Cinft(W_{{\gamma}} \times \R^{k+l})$. Our next goal is to show that $\tilde{a}_f^{\gamma}(x,\xi)$ is a lacunary symbol. To do so, we shall need the following   
\begin{proposition}
\label{prop:1}
Let $(L,\Cinft (G))$ be the left regular representation of $G$. Let $X_{-\lambda,i},H_j $ be the basis elements of $\n^-$ and $\a$ introduced in Section \ref{Sec:2}, and  $(\widetilde W_\gamma, \phi_{\gamma})$ an arbitrary chart. With $x=(n,t)\in W_\gamma$, $\tilde x=\phi_{\gamma} (x) \in \widetilde{W}_\gamma$, $\, g \in V_\gamma^1$ one has
\bq
\label{eq:23}
\begin{pmatrix} dL(X_{-\lambda,1})\psi^\gamma_{\xi,x}(g)\\ \vdots\\ dL(H_l)\psi^\gamma_{\xi,x}(g)\end{pmatrix} =i\psi^\gamma_{\xi,x}(g)\Gamma(x,g)\xi,
\eq
with
\bq
\label{eq:Gamma}
\Gamma(x,g)=
\left(\begin{array}{cc}
\Gamma_1 & \Gamma_2 \\
 \Gamma_3   &\Gamma_4\\
 \end{array}\right) 
=
\left(\begin{array}{cccc}
 dL(X_{-\lambda,i})n_{j,\tilde x}(g)& \multicolumn{1}{c|}{}  & &dL(X_{-\lambda,i})\chi_j(g,\tilde x) \\
&  \multicolumn{1}{c|}{} &  &\\
\cline{1-4} &   \multicolumn{1}{c|}{}  & &\\
 dL(H_i)n_{j,\tilde x}(g) & \multicolumn{1}{c|}{} &  &dL(H_i)\chi_j(g,\tilde x)\\
 \end{array}\right)
\eq
 belonging to $\GL(l+k,\R)$, where  $n_{j,\tilde x}(g)=n_j(g\cdot \tilde x)$.
\end{proposition}
\begin{proof} 
Fix a chart $(\widetilde W_\gamma, \phi_{\gamma})$, and let  $x$, $\tilde x$, $ g$ be as above. For $X\in\g$, one computes that
\begin{align*}
dL(X)\psi^\gamma_{\xi,x}(g)&=\frac{d}{ds}e^{i({\bf 1}_k\otimes T^{-1}_{t})\phi^{\e{-sX}g}_{\gamma} (x)\cdot\xi}|_{s=0}=i\psi^\gamma_{\xi,x}(g)\Big [\sum_{{i}=1}^{k}\xi_{i}dL(X)n_{i,\tilde x}(g)\\ &+\sum_{{j}=k+1}^{l+k}\xi_{j}dL(X)\chi_{j}(g,\tilde x)\Big ],
\end{align*}
showing the first equality. To see the invertibility of the matrix $\Gamma(x,g)$, note that for small $s$
\bqn 
\chi_j(\e{-sX}g, \tilde x) = \chi_j(g,\tilde x) \chi_j (\e{-sX},g\cdot \tilde x).
\eqn
Lemma \ref{lemma:char} then yields 
\begin{align*}
dL(H_i) \chi_j(g,\tilde x)&= \chi_j(g,\tilde x)  \frac d {ds} \Big ( e^{c_{ij}(m_{w_\gamma})s} \Big )_{|s=0} =\chi_j(g,\tilde x)c_{ij}(m_{w_\gamma}).
\end{align*}
 This means that  $\Gamma_4$ is the product of the matrix $\left (c_{ij}(m_{w_\gamma})\right )_{i,j}$ with the diagonal matrix whose $j$-th diagonal entry is $\chi_j(g,\tilde x)$. Since  $\left (c_{ij}(m_{w_\gamma})\right )_{i,j}$ is just the matrix representation of $\Ad(m_{w_\gamma}^{-1}) $ relative to the basis $\{H_1,\dots, H_l\}$ of $\a$, it is invertible. On the other hand, $\chi_j(g,\tilde x)$ is non-zero for all $j \in \{1,\dots,l\}$ and arbitrary $g $ and $ \tilde x$.  Therefore $\Gamma_4$, being the product of two invertible matrices, is invertible. Next, let us show that  the matrix  $\Gamma_1$ is non-singular. Its  $(ij)^{th}$ entry reads
\bqn
dL(X_{-\lambda,i})n_{j,\tilde x}(g)=\frac d{ds} n_{j,\tilde x}(\e{-sX_{-\lambda,i}}\cdot g )_{|s=0}=(-X_{-\lambda,{i}|\widetilde{\X}})_{g\cdot \tilde x} (n_{j}).
\eqn
 For $\Theta\subset\Delta$, $q\in \R^l$, we  define the $k$-dimensional submanifolds
 \bqn 
 \mathfrak{L}_{\Theta}(q)=\{ \tilde x=\phi_{\gamma}(n,q)\in \widetilde W_{\gamma}: q_i\neq 0 \Leftrightarrow \alpha_i \in \Theta \},
  \eqn  
and consider  the decomposition  $T_{g\cdot \tilde x}\widetilde{\X}_{\Theta}=T_{g\cdot \tilde x}\mathfrak{L}_{\Theta}(q)\oplus N_{g\cdot \tilde x}\mathfrak{L}_{\Theta}(q)$  of $T_{g\cdot \tilde x}\widetilde{\X}_{\Theta}$ into the tangent and normal space to $\mathfrak{L}_{\Theta}(q)$ at the point  $g\cdot \tilde x \in \widetilde{\X}_{\Theta}$.
Since  $\widetilde{\X}_{\Theta}$ is a $G$-orbit, the group $G$ acts transitively on it. Now, as $g$ varies over $G$ in Lemma \ref{lemma:fundvec}, one deduces that $N^{-}\times A(\Theta)$ acts locally transitively on $\widetilde{\X}_{\Theta}$. In addition, by the definition of  $\mathfrak{L}_{\Theta}(q)$,  $N_{g\cdot \tilde x}\mathfrak{L}_{\Theta}(q)$ is spanned  by the vector fields $\{ -t_i\frac{\partial}{\partial t_i}\}_{ \alpha_i \in\Theta}$.  Consequently, $T_{g\cdot \tilde x}\mathfrak{L}_{\Theta}(q)$ must be equal to the span of the vector fields $\{X_{-\lambda,{i}|\widetilde{\X}}\}$, which means that   $N^-$ acts locally transitively on $\mathfrak{L}_{\Theta}(q)$ for arbitrary $\Theta$. Since the latter is  parametrized by the coordinates $(n_1,\dots,n_k)$, one concludes that the matrix  $((X_{-\lambda,{i}|\widetilde{\X}})_{g \cdot \tilde x}(n_j))_{ij}$ has full rank. Thus,  $\Gamma_1$ is non-singular. On the other hand, if $\tilde x =\pi(e,n,t)\in \tilde U_e$, Lemma \ref{lemma:char} implies
\bqn 
dL(X_{-\lambda,i}) \chi_j(g,\tilde x)= \chi_j(g,\tilde x)  \frac {d} {ds}\Big ( \chi_j(e^{-s X_{-\lambda,i}}, g \cdot \tilde x)  \Big )_{|s=0} =0, \\
\eqn
showing that $\Gamma_2$ is identically zero, while $\Gamma_4$ is a non-singular diagonal matrix in this case.  Geometrically, this amounts to the fact that the fundamental vector field   corresponding to $H_j$ is transversal to the hypersurface defined by  $t_j=q\in \R\setminus \{0\}$, while  the vector fields corresponding to the Lie algebra elements $X_{-\lambda,i}, H_i, \, i\neq j$, are tangential. We therefore conclude that $\Gamma(x,g)$ is non-singular if $\tilde x \in \widetilde U_e$. But since the different copies $\widetilde \X_{\Theta_{(e,n,t)}}$ of $G/P_{\Theta_{(e,n,t)}} (K)\simeq N^- \times B_{(e,n,t)}\subset N^- \times \R^l \simeq \widetilde U_e$ in $\widetilde \X$ are isomorphic to each other, the same must hold if $\tilde x$ lies in one of the remaining charts $\widetilde U_{m_{w_\gamma}}$, and the assertion of the lemma follows. 

\end{proof}
We can now state the main result of this paper. In what follows, $\{(\widetilde W_\gamma, \phi_{\gamma})\}_{\gamma \in I}$ will always denote the atlas of $\widetilde \X$ constructed above.
\begin{theorem}
\label{thm:3}
Let $\widetilde \X$ be the Oshima compactification of a Riemannian symmetric space $\X\simeq G/K$ of non-compact type, and $f\in \S({G})$ a rapidly decaying function on $G$.  Let further $\mklm{( \widetilde W_\gamma, \phi_\gamma^{-1})}_{\gamma \in I}$ be the atlas of $\widetilde \X$ construced above. Then the operators $\pi(f)$ are locally of the form 
\begin{gather}
\label{20}
  A^\gamma_fu(x)= \int e ^{i x \cdot \xi} a_f^\gamma(x,\xi)\hat u(\xi) \dbar\xi, \qquad u \in \CT(W_\gamma),
\end{gather}
where $a_f^\gamma(x,\xi)=\tilde  a_f^\gamma(x,\xi_1, \dots, \xi_k, x_{k+1} \xi_{k+1}, \dots, \xi_{k+l} x_{k+l})$, and $\tilde  a_f^\gamma(x,\xi) \in \Symsl(W_\gamma \times \R^{k+l}_\xi)$ is given by \eqref{eq:auxsym}.  In particular, the kernel of the operator $A^\gamma_f$ is determined by its restrictions to $W_\gamma^\ast \times W_\gamma^\ast$,  where $W_\gamma^\ast=\{ x=(n,t) \in W_\gamma: t_1 \cdots t_l \not=0\}$, and given by the oscillatory integral
\begin{equation}
\label{20b}
  K_{A_f^\gamma} (x,y)=\int e^{i(x-y) \cdot \xi} a^\gamma_f(x,\xi) \dbar \xi.
\end{equation}
\end{theorem}
As a consequence, we obtain the following  
\begin{corollary}
\label{corollary}
Let $\widetilde \X_\Delta$ be an open $G$-orbit in $\widetilde \X$ isomorphic to $G/K$. Then the  continuous linear operators
\begin{equation*}
\pi(f)_{|\overline{\widetilde \X_\Delta}}:\CT(\overline{\widetilde \X_\Delta}) \longrightarrow \Cinft(\overline{\widetilde \X_\Delta}),
\end{equation*}
 are  totally characteristic pseudodifferential operators of class $\L^{-\infty}_b$ on the manifolds with corners $\overline{\widetilde \X_\Delta}$.
\end{corollary}
\qed

\begin{proof}[Proof of Theorem \ref{thm:3}]
Our considerations will essentially follow the proof of Theorem 4 in  \cite{ramacher06}. 
Let $\Gamma(x,g)$ be the matrix defined in \eqref{eq:Gamma}, and consider its extension as an endomorphism in $\C^1[\R^{k+l}_\xi]$ to the symmetric algebra ${\rm{S}}(\C^1[\R^{k+l}_\xi])\simeq \C[\R^{k+l}_\xi]$.  Since for $x \in W_\gamma$, $g \in V^1_\gamma$, $\Gamma(x,g)$ is invertible, its extension to $ {\rm{S}}^N(\C^1[\R^{k+l}_\xi])$ is also an automorphism for any $N\in\N$. Regarding  the polynomials $\xi_1,\dots,\xi_{k+l}$ as a basis in $\C^1[\R^{k+l}_\xi]$, let us denote the image of the basis vector $\xi_j$ under the endomorphism $\Gamma(x,g)$ by $\Gamma \xi_j$, so that  by \eqref{eq:23}
\begin{align*}
\Gamma \xi_j&= -i  \psi^\gamma_{-\xi,x}(g) dL(X_{-\lambda,j})\psi^\gamma_{\xi,x}(g),  \qquad  1\leq j \leq k, \\
\Gamma \xi_j&= -i \psi^\gamma_{-\xi,x}(g) dL(H_j) \psi^\gamma_{\xi,x}(g),  \qquad  k+1\leq j \leq k+l.
\end{align*}
 Every polynomial $\xi_{j_1} \otimes \dots \otimes \xi_{j_N}\equiv \xi_{j_1} \dots \xi_{j_N}$ can then be written as a linear combination
 \begin{equation}
\label{24}
   \xi^\alpha =\sum _\beta \Lambda^\alpha_\beta (x,g) \Gamma \xi_{\beta_1} \cdots \Gamma \xi_{\beta_{|\alpha|}},
 \end{equation}
where the $\Lambda^\alpha_\beta(x,g)$ are real-analytic functions on $W_\gamma \times V^1_\gamma$. We need now the following 
\begin{lemma}
\label{lem:3}
For arbitrary indices $\beta_1,\dots, \beta_r$, one has
  \begin{align}
\label{25}
\begin{split}
    i^r \psi^\gamma_{\xi,x}(g) \Gamma\xi_{\beta_1} \cdots \Gamma\xi_{\beta_r}&= dL(X_{\beta_1} \cdots X_{\beta_r}) \psi^\gamma_{\xi,x}(g)\\&+ \sum_{s=1}^{r-1} \sum _{\alpha_1,\dots, \alpha_s} d ^{\beta_1,\dots, \beta_r}_{\alpha_1,\dots, \alpha_s} (x ,g) dL(X_{\alpha_1} \cdots X_{\alpha_s}) \psi^\gamma_{\xi,x}(g),
\end{split}  
\end{align}
where the coefficients $ d ^{\beta_1,\dots, \beta_r}_{\alpha_1,\dots, \alpha_s} (x ,g) \in \Cinft(\tilde W_\gamma \times \supp c_\gamma)$ are at most of exponential growth  in $g$, and independent of $\xi$.
\end{lemma}
\begin{proof}
The lemma is proved by  induction.   For $r=1$ one has $i \psi^\gamma_{\xi,x}(g) \Gamma \xi_p =d L(X_p) \psi^\gamma_{\xi,x}(g)$, where $1 \leq p \leq d$. Differentiating the latter equation with respect to $X_j$, and  writing $\Gamma \xi_p = \sum _{s=1}^{k+l} \Gamma_{ps} (x,g) \, \xi_s$, we obtain with \eqref{24} the equality
  \begin{gather*}
    -\psi^\gamma_{\xi,x}(g) \Gamma \xi_j \Gamma \xi_p = dL(X_jX_p) \psi^\gamma_{\xi,x}(g)-\sum_{s,r=1}^{k+l} (dL (X_j) \Gamma_{ps}) (x,g) \Lambda ^s_r(x,g) dL (X_r) \psi^\gamma_{\xi,x}(g).
  \end{gather*}
Hence, the assertion of the lemma is correct for $r=1,2$. Now, assume that it holds for  $r\leq N$. Setting $r=N$ in equation \eqref{25}, and differentiating with respect to $X_p$, yields for the left hand side
\begin{gather*}
  i^{N+1} \psi^\gamma_{\xi,x}(g) \Gamma \xi_p \Gamma\xi_{\beta_1} \cdots \Gamma \xi_{\beta_N} \\+ i^N  \psi^\gamma_{\xi,x}(g) \Big ( \sum_{s,q=1}^{k+l} (dL(X_p) \Gamma_{\beta_1s})(x,g) \Lambda^s_{q} (x,g) \Gamma \xi_q \Big ) \Gamma \xi _{\beta_2} \cdots \Gamma \xi_{\beta_N} + \dots.
\end{gather*}
By assumption, we can apply \eqref{25} to the products $\Gamma \xi_q \Gamma\xi_{\beta_2} \cdots\Gamma \xi_{\beta_N}, \dots$ of at most $N$ factors. Since the functions $n_{i,m}(g)$ and $ \chi_j(g,m)$, and consequently the coefficients of $\Gamma(x,g)$, are at most of exponential growth in $g$, the assertion of  the lemma follows.
\end{proof}

\noindent
\emph{End of proof of Theorem \ref{thm:3}}. 
Let us next show that $\tilde  a_f^\gamma(x,\xi) \in \Syms(W_\gamma \times \R^{k+l}_\xi)$. As already noted, $\tilde  a_f^\gamma(x,\xi) \in \Cinft(W_\gamma \times \R^{k+l}_\xi)$. While differentiation with respect to $\xi$ does not alter the growth properties of $\tilde a^\gamma_f(x,\xi)$, differentiation with respect to $x$  yields additional powers in $\xi$.  
Now,  as an immediate consequence of equations \eqref{24} and \eqref{25}, one computes for arbitrary $N \in \N$
\begin{equation}
  \label{26}
 \psi^\gamma_{\xi,x}(g)(1+\xi^2)^N=  \sum_{r=0}^{2N} \sum_{|\alpha| =r} b^N_\alpha(x,g) d L(X^\alpha) \psi^\gamma_{\xi,x}(g),
\end{equation}
where the coefficients $b^N_\alpha(x,g) \in \Cinft(W_\gamma \times V^1_\gamma)$  are at most of exponential growth in  $g$. Now, $(\gd^\alpha_\xi \gd^\beta_x \tilde  a ^\gamma_f)(x,\xi)$ is a finite sum of terms of the form
\begin{equation*}
  \xi^\delta e^{-i(x_1,\dots x_k,1, \dots, 1) \cdot \xi} \int_{G} f(g) d_{\delta \beta}(x,g) \psi^\gamma_{\xi,x}(g)c_\gamma(g) dg,
\end{equation*}
the functions $d_{\delta\beta}(x,g) \in \Cinft (W_\gamma \times V^1_\gamma)$ being at most of exponential growth in $g$. Making use of equation \eqref{26}, and integrating according to Proposition \ref{prop:A}, we finally obtain for arbitrary $\alpha, \beta$ the estimate
\begin{equation*}
  |(\gd ^\alpha_\xi \gd ^\beta _x \tilde  a_f^\gamma) (x,\xi) | \leq \frac 1 {(1+\xi^2)^N} C_{\alpha,\beta,\mathcal{K}} \qquad x \in \mathcal{K},
\end{equation*}
where $\mathcal{K}$ denotes an arbitrary  compact set  in $W_\gamma$, and $N\in \N$. This proves that $\tilde a^\gamma_f(x,\xi) \in \Syms(W_\gamma \times \R^{k+l}_\xi)$. Since equation \eqref{20} is an immediate consequence of Fourier inversion formula,  it remains to show that $\tilde  a_f^\gamma(x,\xi)$ satisfies the lacunary condition \eqref{V} for each of the coordinates $t_i$. Now, it is clear that $a^\gamma_f \in \mathrm{S}^{-\infty} ( W_\gamma ^\ast \times \R^{k+l}_\xi$), since $G$ acts transitively on each $\widetilde \X_{\Delta}$.  As a consequence, the Schwartz kernel   of the restriction of the operator $A^\gamma_f:\CT (W_\gamma) \rightarrow \Cinft(W_\gamma)$  to $W^\ast_\gamma$ is given by the absolutely convergent integral 
\bqn 
\int e^{i(x-y) \cdot \xi} a^\gamma_f(x,\xi) \dbar \xi \in \Cinft( W_\gamma^\ast \times W_\gamma^\ast).
\eqn
 Next, let us write $W_\gamma=\bigcup _{\Theta \subset \Delta} W_\gamma^\Theta$, where $W_\gamma^\Theta=\mklm{x=(n,t): t_i \not=0 \Leftrightarrow \alpha_i \in \Theta}$. Since on $W_\gamma^\Theta$ the function $A_f^\gamma u$ depends only on the restriction of $u \in \CT(W_\gamma)$ to $W_\gamma^\Theta$, one deduces that 
\bq
\label{eq:33}
\supp K_{A_f^\gamma} \subset \bigcup _{\Theta \subset \Delta} \overline{W_\gamma^\Theta} \times  \overline{W_\gamma^\Theta}.
\eq
Therefore, each of the integrals 
\bqn 
\int e ^{i(x_j-y_j) \xi_j} \tilde a^\gamma _f (x,({\bf{1}}_k \otimes T_x)\xi) \d \xi_j,\qquad  j=k+1, \dots, k+l,
\eqn
which are smooth functions on $W_\gamma^\ast \times W_\gamma^\ast$, must vanish if $x_j$ and $y_j$ do not have the same sign. With the substitution $r_j=y_j/x_j -1$, $\xi_j x_j =\xi_j'$ one finally arrives at the conditions
\bqn 
\int e^{-ir_j \xi_j} \tilde a^\gamma_f(x,\xi) \d \xi_j =0 \qquad \mbox{ for } r_j < -1, \, x \in W_\gamma^\ast.
\eqn
 But since $\tilde a^\gamma_f$ is rapidly decreasing in $\xi$, the Lebesgue bounded convergence theorem implies that these conditions must also hold for $x \in W_\gamma$. Thus, the lacunarity of the symbol $\tilde a_f ^\gamma$ follows.  The fact that the kernel   $K_{A^\gamma_f}$ must be determined by its restriction to $W_\gamma^\ast \times W_\gamma^\ast$, and hence by the oscillatory integral \eqref{20b}, is now a consequence of  \cite{melrose}, Lemma 4.1, completing the proof of Theorem \ref{thm:3}.
\end{proof}

As a consequence of Theorem \ref{thm:3}, we can locally write the kernel of $\pi(f)$ in the form
\begin{align}
\label{27}
\begin{split}
  K_{A_f^\gamma}(x,y) &= \int e ^{i(x-y) \cdot  \xi} a^\gamma _f (x,\xi) \dbar \xi=\int e^{i(x-y) \cdot  ({\bf{1}}_k \otimes T_x^{-1}) \xi} \tilde a_f^\gamma(x,\xi) |\det ({\bf{1}}_k \otimes T_x^{-1})'(\xi)| \dbar \xi\\
&=\frac 1 {|x_{k+1}\cdots x_{k+l}|} \tilde A_f^\gamma(x,x_1-y_1, \dots, 1 - \frac{y_{k+1}}{x_{k+1}}, \dots), \qquad x_{k+1}\cdots x_{k+l} \not=0, 
\end{split}
\end{align}
where $\tilde A_f^\gamma(x,y)$ denotes the inverse Fourier transform of $\tilde a_f^\gamma(x,\xi)$,
\begin{equation}\label{eq:34}
  \tilde A_f^\gamma(x,y)= \int e ^{i y\cdot \xi} \tilde a_f ^\gamma(x,\xi) \, \dbar \xi.
\end{equation}
Since for $x \in W^\gamma$ the amplitude  $\tilde a_f^\gamma(x,\xi)$ is rapidly falling in $\xi$, it follows that $\tilde A_f^\gamma(x,y) \in \S(\R^n_y)$,  the Fourier transform being an isomorphism on the Schwartz space. Therefore $K_{A^\gamma_f}(x,y)$ is rapidly decreasing as $| x_{j}| \to 0 $ if $x_j\not=y_j$ and   $k+1\leq j\leq  k+l$.  Furthermore, by the lacunarity of $\tilde a ^\gamma_f$, $ K_{A_f^\gamma}(x,y)$ is also rapidly decaying  as $| y_{j}| \to 0 $ if $x_j\not=y_j$ and   $k+1\leq j\leq  k+l$.

\section{Holomorphic semigroup and resolvent kernels}
\label{Sec:5}

In this section, we shall study   the holomorphic semigroup generated by  a strongly elliptic operator $\Omega$ associated to the  regular representation $(\pi, \mathrm{C}(\widetilde \X))$ of $G$,   as well as  its resolvent. Both the holomorphic semigroup and the resolvent can be characterized as convolution operators of the type considered before, so that we can study them by the methods developed in the previous section. In particular, this will allow us to  obtain a description of the asymptotic behavior of the semigroup and resolvent kernels on $\widetilde \X_\Delta\simeq \X$ at infinity.   

Let us begin by recalling some basic facts about elliptic operators and parabolic evolution equations on Lie groups, our main reference being \cite{robinson}. Let $\G$ be a Lie group, and $\pi$ a continuous representation of $\G$ on a  Banach space $\B$. Let further $X_1,\dots, X_d$ be a basis of the Lie algebra $\mathrm{Lie}(\G)$ of $\G$, and
\bqn 
\Omega= \sum_{|\alpha| \leq q}  c_\alpha \d \pi (X^\alpha)
\eqn
a \emph{strongly elliptic differential operator of order $q$} associated with  $\pi$, meaning that for all $\xi \in \R^d$ one has the inequality $\Re (-1)^{q/2} \sum_{|\alpha| =q} c_{\alpha} \xi^\alpha \geq \kappa |\xi|^q$  for some $\kappa >0$. By the general theory of strongly continuous semigroups, its closure generates a strongly continuous holomorphic semigroup of bounded operators  given by
\begin{equation*}
S_\tau=\frac 1 {2\pi i} \int _\Gamma e^{\lambda \tau} ( \lambda \1+\overline{\Omega})^{-1} d\lambda,
\end{equation*}
where $\Gamma$ is a appropriate path in $\C$ coming from infinity and going to infinity such that   $\lambda \notin \sigma(\overline{\Omega})$ for  $\lambda \in \Gamma$. Here $|\arg \tau|< \alpha$ for an appropriate $\alpha \in (0,\pi/2]$, and the integral converges uniformly with respect to the operator norm. Furthermore, the subgroup $S_\tau$ can be characterized by a convolution semigroup of complex measures $\mu_\tau$ on  $\G$ according to  
\begin{equation*}
S_\tau=\int_\G \pi(g) d\mu_\tau(g),
\end{equation*}
$\pi$ being measurable with respect to the measures $\mu_\tau$. The measures $\mu_\tau$ are absolutely continuous with respect to Haar measure $d_\G$ on  $\G$, and denoting by $K_\tau (g)\in L^1(\G,d_\G)$ the corresponding Radon-Nikodym derivative, one has
\begin{equation*}
S_\tau=\pi(K_\tau)=\int_\G K_\tau(g) \pi(g) d_\G(g).
\end{equation*}
The function $K_\tau(g)\in \L^1(\G,d_\G)$ is analytic in $\tau$ and $g$,
and universal for all Banach representations. It satisfies the parabolic differential equation 
\bqn
\frac{\gd K_\tau}{\gd \tau} (g) + \sum_{|\alpha| \leq q } c_\alpha \, dL(X^\alpha) K_\tau(g)=0, \qquad \lim_{\tau \to 0 } K_\tau(g) = \delta(g),
\eqn
where $(L, \Cinft(\G))$ denotes the left regular representation of $\G$. As a consequence, $K_\tau$ must be  supported on the identity component $\G_0$ of $\G$. It is called the \emph{Langlands kernel} of the holomorphic semigroup $S_\tau$, and satisfies the following $\L^1$- and $\L^\infty$-bounds. 
\begin{theorem}
\label{thm:4}
For each $\kappa \geq 0$, there exist constants $a,b,c>0$, and  $\omega\geq 0$ such that   
\begin{equation}
\label{eq:35}
\int_{\G_0} |dL(X^\alpha)\gd^\beta_\tau K_\tau(g)| e^{\kappa |g|}\d_{\G_0}(g) \leq a b^{|\alpha|}
c^\beta {|\alpha|}!\, \beta!(1+\tau^{-\beta-{|\alpha|}/q })e^{\omega \tau},
\end{equation}
for all   $\tau>0$, $\beta=0,1,2,\dots$ and multi-indices $\alpha$. Furthermore,
\begin{equation}
\label{eq:36}
|dL(X^\alpha)\gd^\beta_\tau K_\tau(g)|\leq a b^{|\alpha|}
c^\beta {|\alpha|}!\, \beta!(1+\tau^{-\beta-({|\alpha|}+d+1)/q })e^{\omega \tau}e^{-\kappa |g|},
\end{equation}
for all $g \in \G_0$, where $d=\dim \G_0$, and $q$ denotes the order of $\Omega$.
\end{theorem}
 
 A detailed exposition  of these facts can be found in \cite{robinson},  pages 30, 152, 166, and 167.   Let now $\G=G$, and $(\pi,\B)$ be the regular representation of $G$ on  $C(\widetilde \X)$. Theorem \ref{thm:4} implies that  the Langlands kernel $K_\tau$ belongs to the space $\S(G)$ of rapidly falling functions on $G$. As a consequence of the previous considerations we obtain

\begin{theorem}
\label{thm:heatoperator}
Let $\Omega$ be a strongly elliptic differential operator of order $q$ associated with the regular representation $(\pi,C(\widetilde \X))$, and $S_\tau=\pi(K_\tau)$ the holomorphic semigroup of bounded operators generated by $\overline \Omega$. Then the operators $S_\tau$ are locally of the form \eqref{20} with $f$ being replaced by $K_\tau$, and  totally characteristic pseudodifferential operators of class $\L^{-\infty}_b$ on the manifolds with corners $\overline{\widetilde\X_\Delta}$. Furthermore, on $W_\gamma \times W_\gamma$, the  kernel of $S_\tau$ is  given by
\begin{align*}
\begin{split}
  S^\gamma_\tau(x,y)=K_{A_{K_\tau}^\gamma}(x,y) &= \int e ^{i(x-y) \cdot  \xi} a^\gamma _{K_\tau} (x,\xi) \dbar \xi =\frac 1 {|x_{k+1}\cdots x_{k+l}|} \tilde A_{K_\tau}^\gamma(x,({\bf{1}}_k \otimes T_x^{-1})(x-y)),
  \end{split}
\end{align*}
where $x_{k+1}\cdots x_{k+l} \not=0$,  and  $\tilde A_{K_\tau}^\gamma(x,y)$ was  defined in  \eqref{eq:34}. 
 In particular,  $S^\gamma_\tau(x,y)$ is rapidly falling at infinity as $| x_{j}| \to 0 $, or $| y_{j}| \to 0 $, as long as $x_j\not=y_j$, where  $k+1\leq j\leq  k+l$. In addition,   
\bq
\label{eq:38z}
|\tilde A_{K_\tau}^\gamma(x,y)| \leq \begin{cases}
 c_1 (1+\tau^{-(l+k+1)/q}), & 0 < \tau \leq 1,  \\
 c_2 e^{\omega \tau}, & 1< \tau,
 \end{cases}
\eq
 uniformly on compact subsets of $ W_\gamma\times W_\gamma$ for some constants $c_i>0$.
\end{theorem}
\begin{proof} 
The first assertions are  immediate consequences of Theorem \ref{thm:3}, and its corollary. In order to prove  \eqref{eq:38z}, note that for large $N\in \N$ one computes with \eqref{eq:auxsym}, \eqref{26}, and  \eqref{eq:34}
\begin{align*}
|\tilde A_{K_\tau}^\gamma(x,y)|& \leq   \int_{\R^{k+l}}  |\tilde a_{K_\tau} ^\gamma(x,\xi)| \, \dbar \xi =\int _{\R^{k+l}}\Big |\int_{G}\psi^\gamma_{\xi,x}(g) c_{\gamma}(g) K_\tau(g) d_G(g) \Big | \dbar \xi \\
&= \int_{\R^{k+l}} (1+|\xi|^2)^{-N} \Big | \int_G   c_{\gamma}(g) K_\tau(g)\sum_{r=0}^{2N} \sum_{|\alpha| =r} b^N_\alpha(x,g) d L(X^\alpha) \psi^\gamma_{\xi,x}(g)
 d_G(g) \Big | \dbar \xi.
\end{align*}
If we now apply  Proposition \ref{prop:A}, and take into account the estimate \eqref{eq:35} we  obtain
\begin{align*}
|\tilde A_{K_\tau}^\gamma(x,y)|& \leq \int (1+|\xi|^2)^{-N} \Big | \int_G   \psi^\gamma_{\xi,x}(g)  \sum_{r=0}^{2N} \sum_{|\alpha| =r}  d L(X^{\tilde \alpha}) [ b^N_\alpha(x,g)c_{\gamma}(g) K_\tau(g)]
 d_G(g) \Big | \dbar \xi\\
 &\leq \begin{cases}
 c_1 (1+\tau^{-2N/q}), & 0 < \tau \leq 1,  \\
 c_2 e^{\omega \tau}, & 1< \tau,
 \end{cases}
\end{align*}
for certain constants $c_i>0$.  Expressing $\xi^{k+l+1}_j \psi^\gamma_{\xi,x}(g)$ on $\mklm{\xi \in \rn: |\xi_i| \leq |\xi_j| \, \text{for all } \, i}$ as  left derivatives of $\psi^\gamma_{\xi,x}(g)$ according to \eqref{24} and \eqref{25}, and estimating the maximum norm on $\rn$ by the usual norm, 
a similar argument shows that the  last estimate is also valid for  $N=(k+l +1)/2$, compare \eqref{eq:49}. The proof is now complete. 
\end{proof}

Let us now turn to the resolvent of the closure of the strongly elliptic operator $\Omega$. By \eqref{eq:35} one has the bound $\norm{S_\tau}\leq c e^{\omega \tau}$ for some constants $c \geq 1, \omega \geq 0$. For $\lambda \in \C$ with $\Re \lambda > \omega$, the resolvent of $\overline \Omega$ can then be expressed by means of the Laplace transform according to 
\bqn 
(\lambda \1 + \overline{\Omega})^{-1} = {\Gamma(1)}^{-1} \int_0^\infty e^{-\lambda \tau}  S_\tau \d \tau, 
\eqn
where $\Gamma$ is the $\Gamma$-function. 
More generally, one can consider for arbitrary $\alpha >0$ the integral transforms
\bqn 
(\lambda \1 + \overline{\Omega})^{-\alpha} = {\Gamma(\alpha)}^{-1} \int_0^\infty e^{-\lambda \tau} \tau^{\alpha-1} S_\tau \d \tau. 
\eqn
As it turns out, the functions 
\bqn
R_{\alpha, \lambda}(g) = \Gamma(\alpha)^{-1} \int_0 ^\infty e^{-\lambda \tau} \tau^{\alpha-1} K_\tau(g)  \d \tau
\eqn
are in $\L^1(G,e^{\kappa |g|}d_G)$, where $\kappa \geq 0$ is such that $\norm{\pi(g)} \leq c e^{\kappa|g|}$ for some $c\geq 1$. This implies that the resolvent of $\overline \Omega$ can be expressed as the convolution operator
\bqn 
(\lambda \1+ \overline \Omega)^{-\alpha} =\pi(R_{\alpha, \lambda})= \int_G R_{\alpha, \lambda}(g) \pi(g) \d_G(g). 
\eqn
The resolvent kernels $R_{\alpha,\lambda}$ decrease exponentially as $|g| \to \infty$, but they are singular at the identity if $d \geq q \alpha$. More precisely, one has the following
\begin{theorem}
\label{thm:res.est}
There exist constants $b,c, \lambda_0>0$, and $a_{\alpha,\lambda}>0$, such that 
\bqn 
|dL(X^\delta) R_{\alpha, \lambda} (g) | \leq \begin{cases} a_{\alpha, \lambda}  |g|^{-(d+|\delta|-q\alpha)}e^{-(b (\Re \lambda)^{1/q}-c)|g|}, & d>q\alpha, \\a_{\alpha, \lambda} (1+ |\log |g|| ) e^{-(b (\Re \lambda)^{1/q}-c)|g|},& d=q\alpha, \\ a_{\alpha, \lambda}  e^{-(b (\Re \lambda)^{1/q}-c)|g|}, & d < q \alpha
\end{cases}
\eqn
for each $\lambda \in \C$ with $\Re \lambda > \lambda_0$.
\end{theorem}

A proof of these estimates is given in \cite{robinson}, pages 238 and 245. Our next aim is to understand the microlocal structure of the operators $\pi(R_{\alpha, \lambda})$ on the Oshima compactification $\widetilde \X$ 
of $\X\simeq G/K$. Consider again the atlas $\mklm{(\widetilde{W}_\gamma, \phi^{-1}_{\gamma})}_{\gamma \in I}$ of $\widetilde \X$ introduced in Section \ref{Sec:4}, and the local operators 
\bq
\label{eq:38a}
A_{R_{\alpha,\lambda}}^{\gamma} u=[\pi(R_{\alpha,\lambda})_{|\widetilde W_{\gamma}}(u\circ \phi_{\gamma}^{-1})]\circ \phi_{\gamma},
\eq
 where $u\in\CT(W_{{\gamma}})$ and $W_{\gamma}=\phi^{-1}_{\gamma}(\widetilde W_\gamma)$. By the Fourier inversion formula, $A_{R_{\alpha,\lambda}}^\gamma$ is  given by the absolutely convergent integral
\bq
\label{eq:38b}
A_{R_{\alpha,\lambda}}^\gamma u (x)=
\int_{\rn} e^{i x \cdot \xi} a^\gamma_{R_{\alpha,\lambda}} (x,\xi) \hat u(\xi) \dbar \xi,
\eq
where 
\begin{align*}
a^\gamma_{R_{\alpha,\lambda}}(x,\xi)&= \int_{G}e^{i(\phi_\gamma^{g}(x)-x) \cdot\xi} c_\gamma(g)R_{\alpha, \lambda}(g) d_G(g), \\
\tilde a^\gamma_{R_{\alpha,\lambda}}(x,\xi)&= \int_{G}e^{i[(\1_k \otimes T_x^{-1}) (\phi_\gamma^{g}(x)-x)]\cdot \xi} c_\gamma(g)R_{\alpha, \lambda}(g) d_G(g)
\end{align*}
are smooth functions on $W_\gamma \times \R^{k+l}$, 
since $R_{\alpha,\lambda} \in \L^1(G,e^{\kappa |g|}d_G)$, the notation being the same as in Section \ref{Sec:4}. Moreover, in view of the $\L^1$-bound \eqref{eq:35}, the functions $e^{-\lambda \tau} \tau^{\alpha-1} \tilde a ^\gamma_{K_\tau}(x,\xi)$ and $e^{-\lambda \tau} \tau^{\alpha-1}  a ^\gamma_{K_\tau}(x,\xi)$ are integrable  in $\tau$ over $(0, \infty)$, and by Fubini we obtain the equalities 
\begin{align*}
a^\gamma_{R_{\alpha,\lambda}}(x,\xi)&=\Gamma(\alpha)^{-1} \int _0 ^\infty e^{-\lambda \tau} \tau^{\alpha -1} a ^\gamma_{K_\tau}(x,\xi) d\tau,   \\
\tilde a^\gamma_{R_{\alpha,\lambda}}(x,\xi)&= \Gamma(\alpha)^{-1} \int _0 ^\infty e^{-\lambda \tau} \tau^{\alpha -1} \tilde a ^\gamma_{K_\tau}(x,\xi) d\tau.  
\end{align*}
In what follows, we shall describe the microlocal structure of the resolvent $( \lambda \1 + \overline \Omega)^{-\alpha}$ on $\widetilde \X$, and in particular, its kernel. 
\begin{proposition}
\label{prop:3}
Let $Q$ be the largest integer such that $Q < q\alpha$.  Then $ \tilde a^\gamma_{R_{\alpha,\lambda}}(x,\xi)\in \Sym^{-Q}_{la}(W_\gamma \times \R^{k+l})$. That is, for any compactum $\mathcal{K} \subset W_\gamma$, and arbitrary  multi-indices $\beta,\epsilon$ there exist constants $C_{\mathcal{K},\beta,\epsilon}>0$ such that 
\bq
\label{eq:ressym}
|(\gd^{\epsilon}_x \gd_\xi ^{ \beta} \tilde a^\gamma_{R_{\alpha,\lambda}})(x,\xi) | \leq C_{\mathcal{K},\beta,\epsilon} (1+ |\xi|^2)^{(-Q-|\beta|)/2}, \qquad x \in \mathcal{K}, \, \xi \in \R^{k+l},
\eq
and $ \tilde a^\gamma_{R_{\alpha,\lambda}}$ satisfies the lacunary condition \eqref{V} for each of the coordinates $x_j$, $k+1 \leq j \leq k+l$.
\end{proposition}
\begin{proof}
For a fixed a chart chart $(\widetilde W_\gamma, \phi_{\gamma})$  of $\widetilde \X$ we  write  $x=(n,t)\in W_\gamma$,   $\tilde x=\phi_{\gamma} (x) \in \widetilde{W}_\gamma$ as usual. As a consequence of Proposition \ref{prop:1} and Lemma \ref{lem:3} one computes with \eqref{26} for arbitrary $N \in \N$
\begin{gather*}
(\gd_\xi ^{2 \beta} \tilde a^\gamma_{R_{\alpha,\lambda}})(x,\xi)= \int_G e^{i[(\1_k \otimes T_x^{-1}) (\phi_\gamma^{g}(x)-x)] \cdot \xi} [i(\1_k \otimes T_x^{-1}) (\phi_\gamma^{g}(x)-x)]^{2\beta} c_\gamma(g)R_{\alpha, \lambda}(g) d_G(g)\\
=(1 + |\xi|^2) ^{-N} e^{-i(x_1, \dots, x_k,1, \dots, 1) \cdot \xi}   \sum_{r=0}^{2 N} \sum_{|\delta| =r}\int_G b^{N}_\delta(x,g) d L(X^\delta) \psi^\gamma_{\xi,x}(g)\\
 \cdot  [i(\1_k \otimes T_x^{-1}) (\phi_\gamma^{g}(x)-x)]^{2\beta} c_\gamma(g)R_{\alpha, \lambda}(g) d_G(g).
\end{gather*}
Now, $n_r(g\cdot \tilde x)\to n_r(\tilde x)$ and $   \chi_r(g,\tilde x)\to 1$  as $g \to e $,  so that due to the analyticity of the $G$-action on $\widetilde \X$ one deduces 
\bq
\label{eq:37}
|(\1_k \otimes T_x^{-1}) (\phi_\gamma^{g}(x)-x)|=|(n_1(g\tilde x)- n_1, \dots, \chi_1(g\tilde x)-1, \dots ) | =C_\mathcal{K} |g|, \qquad x \in \mathcal{K}.
\eq
Indeed, let
\bqn 
(\zeta_1, \dots, \zeta_d) \mapsto e^{\zeta_1 X_1 + \dots +\zeta_d X_d} =g
\eqn
be canonical coordinates of the first type near the identity $e \in G$. We then have the power expansions
\bq
\label{eq:38}
\chi_r(g,\tilde x) -1 = \sum_{\alpha,\beta,\gamma} c^r_{\alpha,\beta, \gamma} n^\alpha t^\beta \zeta^\gamma, \qquad n_r(g\cdot \tilde x)-n_r(\tilde x) = \sum_{\alpha,\beta,\gamma} d^r_{\alpha,\beta, \gamma} n^\alpha t^\beta \zeta^\gamma,
\eq
where the constant term vanishes, that is, $c^r_{\alpha,\beta,\gamma},\, d^r_{\alpha,\beta,\gamma}=0$ if $|\gamma|=0$. Hence,
$$|n_r(g\cdot \tilde x)-n_r(\tilde x)|, \, |\chi_r(g,\tilde x) -1|\leq C_1 |\zeta| \leq C_2  |g|,$$
 compare \cite{robinson}, pages 12-13, and we obtain \eqref{eq:37}. With Theorem \ref{thm:res.est}, we therefore have the pointwise estimates
\bqn 
| [(\1_k \otimes T_x^{-1}) (\phi_\gamma^{g}(x)-x)]^{\beta'} dL(X^{\delta'}) R_{\alpha, \lambda} (g) | \leq C_{\mathcal{K}, \alpha, \lambda}  |g|^{-(d+|\delta'|-q\alpha-|\beta'| )}e^{-(b (\Re \lambda)^{1/q}-c)|g|}
\eqn
for some constant $C_{\mathcal{K},\alpha, \lambda}>0$ uniformly on $\mathcal{K}\times V_\gamma^1$. Now, let $ 2\tilde Q$ be the largest even number strictly smaller than $q\alpha$. Applying the same reasoning as in the proof of Proposition \ref{prop:A}, one obtains for $N= \tilde Q+|\beta|$ 
\begin{gather*}
(\gd_\xi ^{2 \beta} \tilde a^\gamma_{R_{\alpha,\lambda}})(x,\xi) 
=(1 + |\xi|^2) ^{-\tilde Q-|\beta|}   \sum_{r=0}^{2\tilde Q+2|\beta|} \sum_{|\delta| =r}(-1)^{|\delta|} \int_G e^{i[(\1_k \otimes T_x^{-1}) (\phi_\gamma^{g}(x)-x)] \cdot \xi}   \\ \cdot  d L(X^{\tilde \delta})\big [  b^{\tilde Q +|\beta|}_\delta(x,g)  [i(\1_k \otimes T_x^{-1}) (\phi_\gamma^{g}(x)-x)]^{2\beta} c_\gamma(g)R_{\alpha, \lambda}(g) \big ]  d_G(g),
\end{gather*}
since all the occuring combinations $ [(\1_k \otimes T_x^{-1}) (\phi_\gamma^{g}(x)-x)]^{\beta'} dL(X^{\delta'}) R_{\alpha, \lambda} (g)$ on the right hand side are such that $q\alpha +|\beta'|  -|\delta'|>0$, implying that  the corresponding integrals  over $G$ converge. Equality then follows by the left-invariance of $d_G(g)$, and Lebesgue's Theorem on Dominated Convergence. To show the estimate \eqref{eq:ressym} in  general for  $\epsilon=0$, let  $x \in \mathcal{K}$, and $\xi \in \R^{k+l}$ be such that $|\xi| \geq 1$, and $|\xi|_{\mathrm{max}} = \max \mklm{|\xi_r|: 1 \leq r \leq k+l}= |\xi_j|$. 
Using \eqref{24} and \eqref{25} we can express $\xi^{Q +|\beta|}_j \psi^\gamma_{\xi,x}(g)$  as left derivatives of $\psi^\gamma_{\xi,x}(g)$, and  repeating the previous argument  we obtain  the  estimate 
\begin{align}
\begin{split}
\label{eq:49}
|(\gd_\xi ^{ \beta} \tilde a^\gamma_{R_{\alpha,\lambda}})(x,\xi)| &= |\xi_j| ^{-Q-|\beta|} \Big | \sum_{r=0}^{Q + |\beta|} \sum_{|\delta| =r}\int_G b^{j}_\delta(x,g) d L(X^\delta) \psi^\gamma_{\xi,x}(g)\\
 \cdot  [i(\1_k \otimes T_x^{-1}) (\phi_\gamma^{g}(x)-x&)]^{\beta} c_\gamma(g)R_{\alpha, \lambda}(g) d_G(g) \Big |\leq  \tilde C_{\mathcal{K},\beta} \frac 1 {|\xi|_{\mathrm{max}}^{Q+ |\beta|}} \leq  C_{\mathcal{K}, \beta} \frac 1 {|\xi|^{Q+ |\beta|}}, 
\end{split}
\end{align}
where the coefficients $b^{j}_\delta(x,g)$ are at most of exponential growth in $g$. But since $\tilde a^\gamma_{R_{\alpha,\lambda}}(x,\xi)\in \Cinft(W_\gamma \times \R^{k+l})$, we obtain \eqref{eq:ressym} for $\epsilon =0$. 
Let us now turn to the $x$-derivatives. We have to show that the powers in $\xi$ that arise when  differentiating $(\gd^{ \beta} _\xi \tilde a^\gamma_{R_{\alpha,\lambda}})(x,\xi) $ with respect to $x$ can be compensated by an argument similar to the previous considerations. Now, \eqref{eq:38} clearly implies 
\bqn 
{\gd_x^\epsilon} \, (\chi_r(g,\tilde x) -1)=O(|g|), \qquad 
{\gd_x^\epsilon}\,  (n_r(g\cdot \tilde x)-n_r(\tilde x) )=O(|g|).
\eqn
Thus, each time we differentiate the exponential $e^{i[(\1_k \otimes T_x^{-1}) (\phi_\gamma^{g}(x)-x)] \cdot \xi}$ with respect to $x$, the result is of order $O(|\xi||g|)$. Therefore, expressing the ocurring powers $\xi^{\epsilon'}  \psi^\gamma_{\xi,x}(g)$ as left derivatives of $\psi^\gamma_{\xi,x}(g)$, we can repeat the preceding argument to absorb the powers in $\xi$, and \eqref{eq:ressym} follows.
Note next that  the previous argument also implies  $ a^\gamma_{R_{\alpha,\lambda}}(x,\xi) \in \mathrm{S}^{-Q} ( W_\gamma ^\ast \times \R^{k+l}_\xi$), where $W_\gamma^\ast=\mklm{x=(n,t) \in W_\gamma: t_1\cdots t_l\not=0}$, the $G$-action being transitive on each $\widetilde \X_\Delta$. The Schwartz kernel  $K_{A_{R_{\alpha,\lambda}}^\gamma}$ of the restriction of the operator \eqref{eq:38a}  to $W^\ast_\gamma$ is therefore given by the oscillatory integral 
\bqn 
\int e^{i(x-y) \cdot \xi} a^\gamma_{R_{\alpha,\lambda}}(x,\xi) \dbar \xi \in \D'( W_\gamma^\ast \times W_\gamma^\ast),
\eqn
which is $\Cinft$ off the diagonal. As in \eqref{eq:33} we have 
$
\supp K_{A_{R_{\alpha,\lambda}}^\gamma} \subset \bigcup _{\Theta \subset \Delta} \overline{W_\gamma^\Theta} \times  \overline{W_\gamma^\Theta}
$,
so that each of the integrals 
\bqn 
\int e ^{i(x_j-y_j) \xi_j} \tilde a^\gamma_{R_{\alpha,\lambda}} (x,({\bf{1}}_k \otimes T_x)\xi) \d \xi_j,\qquad  j=k+1, \dots, k+l,
\eqn
must vanish if $x_j$ and $y_j$ do not have the same sign. Hence, 
\bqn 
\int e^{-ir_j \xi_j} \tilde a^\gamma_{R_{\alpha,\lambda}}(x,\xi) \d \xi_j =0 \qquad \mbox{ for } r_j < -1, \, x \in W_\gamma^\ast.
\eqn
Since $\tilde a^\gamma_{R_{\alpha,\lambda}}(x,\xi) \in \mathrm{S}^{-Q} ( W_\gamma  \times \R^{k+l}_\xi)$, these integrals are absolutely convergent for $r_j\not=0$. Lebesgue's Theorem on Bounded Convergence theorem then implies that these conditions must also hold for $x \in W_\gamma$. The proof of the proposition is now complete. 
\end{proof}

\begin{remark}
One would actually expect that $ \tilde a^\gamma_{R_{\alpha,\lambda}}(x,\xi)\in \Sym_{la}^{-q\alpha}(W_\gamma \times \R^{k+l})$, being the local symbol of the resolvent $(\lambda \1 +\overline \Omega)^{-\alpha}$.  Nevertheless,  the general estimates of Theorem \ref{thm:res.est} for the resolvent kernels $R_{\alpha,\lambda}$, which correctly reflect the singular behavior at the identity, are not sufficient to show this, and  more information about them is required. Indeed, $dL(X^\beta) R_{\alpha,\lambda} \in L_1(G,d_G(g))$ only holds if $0< q\alpha - |\beta|$.  
\end{remark}

We are now able to describe the microlocal structure of the resolvent $(\lambda\1+ \overline \Omega)^{-\alpha}$. 

\begin{theorem}
\label{thm:resolvent}
Let $\Omega$ be a strongly elliptic differential operator of order $q$ associated with the representation $(\pi,C(\widetilde \X))$ of $G$. Let  $\omega\geq 0$ be given by   Theorem \ref{thm:4},     and $\lambda \in \C$ be such that $\Re \lambda > \omega$. Let further $\alpha>0$, and denote by  $Q $ the largest integer such that $Q < q\alpha$.  Then $(\lambda\1+ \overline \Omega)^{-\alpha}=\pi(R_{\alpha,\lambda})$ is locally of the form  \eqref{eq:38b}, where $ a^\gamma_{R_{\alpha,\lambda}} (x,\xi)= \tilde a^\gamma_{R_{\alpha,\lambda}} (x,(\1_k \otimes T_x) \xi)$, and $ \tilde a^\gamma_{R_{\alpha,\lambda}}(x,\xi)\in \Sym_{la}^{-Q}(W_\gamma \times \R^{k+l})$. In particular, $(\lambda\1+ \overline \Omega)^{-\alpha}$ is  a  totally characteristic pseudodifferential operators of class $\L^{-Q}_b$ on the manifolds with corners $\overline{\widetilde\X_\Delta}$. Furthermore, its kernel is locally given by the oscillatory integral
\begin{align*}
\begin{split}
  R^\gamma_{\alpha,\lambda}(x,y)&= \int e ^{i(x-y)  \xi} a^\gamma _{R_{\alpha,\lambda}} (x,\xi) \dbar \xi =\frac 1 {|x_{k+1}\cdots x_{k+l}|}  \int e ^{i({\bf{1}}_k \otimes T_x^{-1})(x-y) \cdot\xi} a^\gamma _{R_{\alpha,\lambda}} (x,\xi) \dbar \xi,
  \end{split}
\end{align*}
where $x_{k+1}\cdots x_{k+l} \not=0, \, x , y \in W_\gamma$. $  R^\gamma_{\alpha,\lambda}(x,y)$ is smooth  off the diagonal, and  rapidly falling at infinity as $| x_{j}| \to 0 $, or $| y_{j}| \to 0 $, as long as $x_j\not=y_j$, where  $k+1\leq j\leq  k+l$. 
\end{theorem}
\begin{proof}
The assertions of the theorem are direct consequences of  our previous considerations, except for  the behavior  of $R^\gamma_{\alpha,\lambda}(x,y)$ at infinity. Let $k+1\leq j\leq  k+l$. While the behavior as $|y_j| \to 0$ is a direct consequence of the lacunarity of $\tilde a^\gamma_{R_{\alpha,\gamma}}$, the behavior as $|x_j| \to 0$ is a direct consequence of the fact that, as oscillatory integrals,
\bqn 
 \int e ^{i(x-y) \cdot  \xi} a^\gamma _{R_{\alpha,\lambda}} (x,\xi) \dbar \xi= \frac 1 {|x-y|^{2N}}  \int e ^{i(x-y) \cdot  \xi} ( \gd_{\xi_1}^2 + \dots +\gd_{\xi_{k+l}}^2  )^N a^\gamma _{R_{\alpha,\lambda}} (x,\xi) \dbar \xi,
\eqn
where $x\not = y$, and $N$ is arbitrarily large. 
\end{proof}

\begin{remark}
The singular behavior of $R_{\alpha,\lambda}(g)$ at the identity corresponds to  the fact that, as a pseudodifferential operator of class $L^{-Q}_b$, $(\lambda \1 + \overline \Omega)^{-\alpha}$ has a kernel which is  singular at the diagonal.  
\end{remark}

To conclude, let us say some words about   the classical heat kernel on a Riemannian symmetric space of non-compact type. 
Consider thus   the regular representation $(\sigma,\mathrm{C}(\widetilde \X))$  of the solvable Lie group $S=AN^-\simeq \X\simeq G/K$ on the Oshima compactification $\widetilde \X$ of $\X$, 
 and associate to every $f\in \S(S)$ the corresponding convolution operator
\bqn 
\int _{S} f(g)  \sigma(g)  \d_{S}(g).
\eqn
 Its restriction to $\Cinft(\widetilde \X)$ induces again a continuous linear operator            
\begin{equation*}
\sigma(f):\Cinft(\widetilde \X) \longrightarrow \Cinft(\widetilde \X) \subset \D'(\widetilde \X),
\end{equation*}
and an examination of the arguments in Section \ref{Sec:4} shows that an analogous analysis applies to the operators $\sigma(f)$. In particular, Theorem \ref{thm:3} holds for them, too. 
Let  $\rho$ be the half sum of all positiv roots, and 
\bqn 
{C}=\sum_j H_j^2 - \sum_{j} Z_j^2-\sum_j [ X_j \theta(X_j) + \theta(X_j) X_j]\equiv \sum_j H_j^2 -2\rho +2 \sum X_j^2 \mod \U(\g) \k
\eqn
be the Casimir operator in $\U(\g)$, where $\mklm{H_j}$, $\mklm{Z_j}$, and  $\mklm{X_j}$ are orthonormal basis of $\a$, $\m$, and $\n^-$, respectively,  and put $C'=\sum_j H_j^2 -2\rho +2 \sum X_j^2$. Though $-d\pi(C')$ is not a strongly elliptic operator in the sense defined above, $\Omega=-d\sigma(C')$ certainly is. Consequently, if $K'_\tau(g) \in \S(S)$ denotes the corresponding Langlands kernel, Theorems \ref{thm:heatoperator} and \ref{thm:resolvent} yield descriptions of the Schwartz kernels of $\sigma(K'_\tau)$ and $(\lambda \1 + \overline \Omega)^{-\alpha}$  on $\widetilde \X$. On the other hand, denote by $\Delta$  the Laplace-Beltrami operator on $\X$. Then
\bqn 
\Delta \phi(gK)=\phi(g:C)= \phi(g:C'), \qquad \phi \in \Cinft(\X),
\eqn
and the associated heat kernel  $h_\tau(g)$ on $\X$ coincides with the heat kernel on $S$ associated to $C'$. But the latter is essentially given by the Langlands kernel $K'_\tau(g)$, being the solution of the parabolic equation
 \bqn
\frac{\gd K'_\tau}{\gd \tau} (g) - dL(C') K'_\tau(g)=0, \qquad \lim_{\tau \to 0 } K'_\tau(g) = \delta(g)
\eqn 
on $S$. In this particular case, optimal upper and lower bounds for $h_\tau$ and the Bessel-Green-Riesz kernels were 
 given in \cite{anker-ji99} using spherical analysis under certain restrictions coming from the lack of control in the
 Trombi-Varadarajan expansion for spherical functions along the walls. Our asymptotics for the kernels of
  $\sigma(K'_\tau)$ and $(\lambda \1 + \overline \Omega)^{-\alpha}$ on $\widetilde \X_\Delta\simeq \X$ are free 
of restrictions, and in concordance with those of \cite{anker-ji99},  though, of course, less explicit.
A detailed description of the resolvent of $\Delta$ on $\X$ was given in \cite{mazzeo-melrose87}, \cite{mazzeo-vasy05}.


\providecommand{\bysame}{\leavevmode\hbox to3em{\hrulefill}\thinspace}
\providecommand{\MR}{\relax\ifhmode\unskip\space\fi MR }
\providecommand{\MRhref}[2]{%
  \href{http://www.ams.org/mathscinet-getitem?mr=#1}{#2}
}
\providecommand{\href}[2]{#2}



\begin{thebibliography}{10}

\bibitem{anker-ji99}
J.-P. Anker and L.~Ji, \emph{Heat kernel and {G}reen function estimates on
  noncompact symmetric spaces}, Geom. Funct. Anal. \textbf{9} (1999), no.~6,
  1035--1091.

\bibitem{borel-ji}
A.~Borel and L.~Ji, \emph{Compactifications of symmetric and locally symmetric
  spaces}, Birkh\"auser Boston Inc., Boston, 2006.

\bibitem{hoermanderIII}
L.~H{\"{o}}rmander, \emph{The analysis of linear partial differential
  operators}, vol. III, Springer--Verlag, Berlin, Heidelberg, New York, 1985.

\bibitem{loya98}
P.~Loya, \emph{On the b-pseudodifferential calculus on manifolds with corners},
  PhD thesis, 1998.

\bibitem{mazzeo-melrose87}
R.~R. Mazzeo and R.~B. Melrose, \emph{{Meromorphic extension of the resolvent
  on complete spaces with asymptotically constant negative curvature}}, J.
  Funct. Anal. \textbf{75} (1987), 260--310.

\bibitem{mazzeo-vasy05}
R.~R. Mazzeo and A.~Vasy, \emph{Analytic continuation of the resolvent of the
  {Laplacian} on symmetric spaces of noncompact type}, J. Funct. Anal.
  \textbf{228} (2005), 311 – 368.

\bibitem{melrose}
R.~Melrose, \emph{Transformation of boundary problems}, Acta Math. \textbf{147}
  (1982), 149--236.

\bibitem{oshima78}
T.~Oshima, \emph{A realization of {R}iemannian symmetric spaces}, J. Math. Soc.
  Japan \textbf{30} (1978), no.~1, 117--132.

\bibitem{ramacher06}
P.~Ramacher, \emph{Pseudodifferential operators on prehomogeneous vector
  spaces}, Comm. Partial Diff. Eqs. \textbf{31} (2006), 515--546.

\bibitem{robinson}
D.~W. Robinson, \emph{Elliptic operators and {Lie} groups}, Oxford University
  Press, Oxford, 1991.

\bibitem{shubin}
M.~A. Shubin, \emph{Pseudodifferential operators and spectral theory}, 2nd
  edition, Springer--Verlag, Berlin, Heidelberg, New York, 2001.

\bibitem{warner72}
G.~Warner, \emph{Harmonic analysis on semi-simple {Lie} groups}, vol.~I,
  Springer--Verlag, Berlin, Heidelberg, New York, 1972.

\end{thebibliography}
\end{document}